\documentclass[a4paper,12pt]{article}
\usepackage[english]{babel}
\usepackage[utf8]{inputenc}
\usepackage{amssymb}
\usepackage{amsmath}
\usepackage{color}

\usepackage{graphicx}
\usepackage{epstopdf}

\usepackage{verbatim}

\usepackage{hyperref}

\newcommand{\open}[1]{\smallskip\noindent\fbox{\parbox{\textwidth}{\color{blue}\bfseries\begin{center}
      #1 \end{center}}}\\ \smallskip}

\newcommand{\red}[1]{{\color{red}#1}}
\newcommand{\dint}{\,\mathrm{d}}
\newcommand{\einf}{\mathop{\mathrm{ess\,inf}}}
\newcommand{\esup}{\mathop{{\mathrm{ess\,sup}}}}
\newcommand{\ignore}[1]{}

\def\cA{\mathcal{A}}

\def\cE{\mathcal{E}}
\def\cF{\mathcal{F}}

\def\cP{\mathcal{P}}

\def\C{{\mathbb C}}

\def\N{{\mathbb N}}

\def\R{{\mathbb R}}

\def\pv{\overrightarrow{p}}
\def\qv{\overrightarrow{q}}

\def\n{\nonumber}

\def\egX{\cE_{\mathsf{G}}^{X}}
\def\envg{\mathfrak{E}_{\mathsf{G}}}
\def\uGindex{u_{\mathsf{G}}^{X}}
\def\uGindexv{u_{\mathsf{G}}^}
\def\egv{\cE_{\mathsf{G}}^}

\author{Dorothee D. Haroske 
\\ Institute of Mathematics \\
Friedrich Schiller University Jena \\
07737 Jena, Germany \\
e-mail: dorothee.haroske@uni-jena.de\\
\and Cornelia Schneider\thanks{The work of this author has been supported by Deutsche Forschungsgemeinschaft (DFG), grant SCHN 1509/1-2.}\\
Department of Mathematics\\
Friedrich Alexander University Erlangen-Nürnberg\\
91058 Erlangen, Germany \\
e-mail: cornelia.schneider@math.fau.de
\and Kristóf Szarvas\thanks{This research was supported by DAAD Research Grants - One-Year-Grants, 2019/20 (57440918).}
\\ Institute of Mathematics \\
Friedrich Schiller University Jena \\
07737 Jena, Germany \\
e-mail: kristof.szarvas@uni-jena.de}

\title{Growth envelopes of some variable and mixed function spaces
}

\date{}

\newtheorem{thm}{Theorem}[section]
\newtheorem{dfn}[thm]{Definition}
\newtheorem{prop}[thm]{Proposition}
\newtheorem{cor}[thm]{Corollary}
\newtheorem{lem}[thm]{Lemma}
\newtheorem{example}[thm]{Example}
\newtheorem{remark}[thm]{Remark}

\newenvironment{proof}{\begin{trivlist} \item[] \textbf{Proof.}}{\quad \rule{2mm}{2mm} \end{trivlist}}

\begin{document}

\maketitle

\begin{abstract}
We study unboundedness properties of functions belonging Lebesgue and Lorentz spaces with variable and mixed norms using growth envelopes. Our results extend the ones for the corresponding classical spaces in a natural way. 
In the case of   spaces with mixed norms it turns out that  the unboundedness  in the worst direction, i.e., in the direction where $p_{i}$ is the smallest, is crucial. More precisely, the growth envelope is given by  $\envg(L_{\pv}(\Omega)) = (t^{-1/\min\{p_{1}, \ldots, p_{d} \}},\min\{p_{1}, \ldots, p_{d} \})$ for mixed Lebesgue and $\envg(L_{\pv,q}(\Omega)) = (t^{-1/\min\{p_{1}, \ldots, p_{d} \}},q)$ for mixed Lorentz spaces, respectively. 

For the variable Lebesgue spaces  we obtain $\envg(L_{p(\cdot)}(\Omega)) = (t^{-1/p_{-}},p_{-})$,  where $p_{-}$ is  the essential infimum of $p(\cdot)$,  subject to some further assumptions. Similarly, for the variable Lorentz space  it holds 
$\envg(L_{p(\cdot),q}(\Omega)) = (t^{-1/p_{-}},q)$. 


The growth envelope is used for Hardy-type inequalities and limiting embeddings. In particular, as a by-product we determine the smallest classical Lebesgue (Lorentz) space which contains a fixed  mixed or  variable Lebesgue (Lorentz) space, respectively. 

\end{abstract}



\section{Introduction}
\label{Introduction}

Using Sobolev embeddings, the integrability properties of a real function can be deduced from those of its derivatives. Sobolev's famous embedding theorem \cite{Sobolev1963} says, that for $1 \leq p < \infty$ and $k \in \N$, the embedding $W_{p}^k(\Omega) \hookrightarrow L_{r}(\Omega)$ holds for all $1 \leq r \leq \infty$ such that $k < d/p$ and $k/d - 1/p  \geq - 1/r$, where $\Omega \subset \R^d$ is a bounded domain with sufficiently smooth boundary. In the limiting case, when $k = d/p$, we have the embedding $W_{p}^{d/p}(\Omega) \hookrightarrow L_{r}(\Omega)$ only for finite $r$. It can be understood as the impossibility of specifying integrability conditions of a function $f \in W_{p}^{d/p}(\Omega)$ merely by means of $L_{r}$ conditions. Refinements of the Sobolev embeddings in the limiting case were investigated in \cite{Peetre1966,Trudinger1967,Moser1971,Strichartz1972} and the embedding $W_{p}^{d/p}(\Omega) \hookrightarrow L_{\infty,p}(\log L)_{-1}(\Omega)$ was obtained (see \cite{Hansson1979,Brezis1980}), where $1 < p < \infty$.

The Sobolev embeddings were extended replacing the Sobolev spaces $W_{p}^{d/p}$ by the more general Bessel potential spaces $H_{p}^{d/p}$, or by the well-known Besov spaces $B_{p,q}^{d/p}$ or Triebel-Lizorkin spaces $F_{p,q}^{d/p}$, respectively. It is known that the space $B_{p,q}^{d/p}$ contains essentially unbounded functions, if and only if, $0 < p \leq \infty$ and $1 < q \leq \infty$. Naturally arises the question, what can be said in this case about the growth of functions from $B_{p,q}^{d/p}$. For Bessel potential spaces, Edmunds and Triebel \cite{Edmunds1999} proved that the space $H_{p}^{d/p}$ can be characterized by sharp inequalities and the non-increasing rearrangement function $f^*$ of $f$: let $\varkappa$ be a continuous, decreasing function on $(0,1]$ and $1 < p < \infty$. Then the inequality
\begin{eqnarray*}
\left( \int_{0}^{1} \left( \frac{f^*(t) \, \varkappa(t)}{|\ln(t)|} \right)^p \, \frac{\dint t}{t} \right)^{1/p} \leq c \, \|f\|_{H_{p}^{d/p}}
\end{eqnarray*}
holds for some constant $c > 0$ and for all $f \in H_{p}^{d/p}$, if and only if $\varkappa$ is bounded.

The idea of the growth envelopes come from Edmunds and Triebel \cite{Edmunds1999} and appears first in Triebel's monograph \cite{Triebel2001}. The concept was studied in detail by Haroske \cite{Haroske2002,Haroske2007,Haroske2009}. Starting from the previous characterization of $H_{p}^{d/p}$, to investigate the unboundedness of functions on $\R^d$ belonging to the quasi-normed function space $X$, the growth envelope function
\begin{eqnarray*}
\egX(t) := \sup \left\{ f^{*}(t) : \|f\|_{X} \leq 1  \right\}, \qquad t > 0
\end{eqnarray*}
and the additional index $\uGindex \in (0,\infty]$ have been introduced. The latter index gives a finer description of unboundedness and is defined as the infimum of those $v>0$, for which the inequality
\begin{eqnarray*}\label{add index def}
\left( \int_{0}^{\varepsilon} \left(\frac{f^*(t)}{\egv{X}(t)} \right)^{v} \, \mu_{\mathsf{G}}(\dint t) \right)^{1/v} \leq C \, \|f\|_{X}
\end{eqnarray*}
holds for all $f \in X$. Here $\mu_{\mathsf{G}}$ is the Borel measure associated with $1/\egv{X}$. The pair $\envg(X) := \left(\egX,\uGindex \right)$ is called the growth envelope of the function space $X$. In case of classical Lorentz spaces $L_{p,q}(\R^d)$ with $0 < p,q \leq \infty$, with $q =\infty$ if $p=\infty$, the result reads as
$$
\envg\left(L_{p,q}(\R^d)\right) = \left( t^{-1/p}, q \right), \qquad t > 0.
$$

One generalization of the classical Lebesgue space $L_{p}$ is the mixed Lebesgue space $L_{\pv}$, where $\pv = (p_{1},\ldots,p_{d})$ is a vector with positive coordinates. The $\|\cdot\|_{\pv}$-quasi-norm of the function $f$ is defined by
$$
\left\|f\right\|_{\pv} := \left( \int_{\Omega_{d}} \ldots \left( \int_{\Omega_{1}} \left| f(x_{1},\ldots,x_{d})\right|^{p_{1}} \dint x_{1} \right)^{p_{2}/p_{1}} \ldots \dint x_{d} \right)^{1/p_{d}},
$$
where $f$ is defined on $\Omega$, which is the Descartes product of the sets $\Omega_{i}$. These spaces were introduced by Benedek and Panzone and some basic properties of these spaces were proved in \cite{Benedek1961}. For some $0 < p < \infty$, $\pv = (p, \ldots, p)$ we get back the classical Lebesgue space $L_{p}$. Moreover, the mixed Lorentz space $L_{\pv,q}$ will be defined by the quasi-norm
$$
\|f\|_{L_{\pv,q}} := \left( \int_{0}^{\infty} u^q \left\| \chi_{\{ |f| > u \}} \right\|_{\pv}^q \, \frac{\dint u}{u} \right)^{1/q},
$$
where $\pv = (p_{1},\ldots,p_{d})$ is a vector and $0 < q < \infty$ is a number. Here we use the notation $\chi_A$ for the characteristic function of a set $A$. This approach can be seen as a generalization of the classical Lorentz space $L_{p,q}$. It will turn out, that if the measure of $\Omega$ is finite, then for the growth envelopes we have
\begin{eqnarray*}
\envg(L_{\pv}(\Omega)) &=& \left( t^{-1/\min\{p_{1}, \ldots, p_{d} \}}, \min\{p_{1}, \ldots, p_{d} \} \right), \\
\envg(L_{\pv,q}(\Omega)) &=& \left( t^{-1/\min\{p_{1}, \ldots, p_{d} \}}, q \right),
\end{eqnarray*}
see Corollaries \ref{cor env mixleb} and \ref{cor env mixlor} below. We see that in the case of the mixed Lebesgue and Lorentz spaces, the unboundedness in the worse direction, i.e., in the direction, where $p_{i}$ is the smallest, is crucial.

Moreover, we deal with growth envelopes of variable function spaces. Replacing the constant exponent $p$ in the classical $L_{p}$-norm by an exponent function $p(\cdot)$, the variable Lebesgue space $L_{p(\cdot)}$ is obtained. The space $L_{p(\cdot)}$ consists of the functions $f$, whose quasi-norm
$$
\|f\|_{p(\cdot)} := \inf \left\{ \lambda > 0  : \int_{\Omega} \left( \frac{|f(x)|}{\lambda} \right)^{p(x)} \dint x \leq 1 \right\}
$$
is finite and $\Omega\subset\R^d$. These spaces were introduced by Kov\'a\v{c}ik and R\'akosn\'{\i}k \cite{Kovacik1991} in 1991, where some of their properties were investigated. From this starting point a lot of research has been undertaken regarding this topic. We refer, in particular, to the monographs by Diening et al. \cite{Diening2011} and Cruz-Uribe and Fiorenza \cite{Cruz-Uribe2013}. The variable Lebesgue spaces are used for variational integrals with non-standard growth conditions \cite{Acerbi2001,Zhikov1987,Zhikov1995}, which are related to modeling of so-called electrorheological fluids \cite{Rajagopal1996,Rajagopal2001,Rruzicka2000}. These spaces are widely used in the theory of harmonic analysis, partial differential equations \cite{Cruz-Uribe2011,Cruz-Uribe2006,Diening2004,Diening2009}, moreover in fluid dynamics and image processing \cite{Acerbi2002a,Acerbi2002,Diening2002,Fan2007,Ruzicka2004}, as well.

The variable Lorentz space $L_{p(\cdot),q}$ will be defined in this paper, where $p(\cdot)$ is an exponent function and $q$ is a number. The measurable function $f : \Omega \to \R$ belongs to the space $L_{p(\cdot),q}$, if
$$
\|f\|_{L_{p(\cdot),q}} = \left( \int_{0}^{\infty} u^q \left\| \chi_{\{ |f| > u \}} \right\|_{p(\cdot)}^q \, \frac{\dint u}{u} \right)^{1/q}
$$
is finite.

In this paper we will study the growth envelope of the spaces $L_{p(\cdot)}$ and $L_{p(\cdot),q}$. We will show in Corollaries \ref{cor old} and \ref{cor old-2}, that subject to some restrictions for small $t>0$,
\begin{eqnarray*}
\egv{L_{p(\cdot)}}(t) &=& \sup \left\{ \|\chi_{A}\|_{p(\cdot)}^{-1} : \mbox{measure of $A$ is equal to $t$} \right\},\\
\egv{L_{p(\cdot),q}}(t) &=& \sup \left\{ \|\chi_{A}\|_{L_{p(\cdot),q}}^{-1} : \mbox{measure of $A$ is equal to $t$} \right\}.
\end{eqnarray*}
If, additionally, the so-called locally log-Hölder continuity for the exponent function $p(\cdot)$ is assumed, the growth envelope function of $L_{p(\cdot)}$ and $L_{p(\cdot),q}$ can be written in the form
$$
\egv{L_{p(\cdot)}}(t) \sim t^{-1/p_{-}} \quad \mbox{and} \quad \egv{L_{p(\cdot),q}}(t) \sim t^{-1/p_{-}}, \quad 0 < t < \varepsilon,
$$
where $p_{-}$ denotes the essential infimum of the exponent function $p(\cdot)$, see Corollaries~\ref{cor equiv leb} and \ref{cor equiv lor} below. Here and in what follows the symbol $f \sim g$ means for positive functions $f$ and $g$, that there are positive constants $A$ and $B$ such that for all $t$,  $A \, g(t) \leq f(t) \leq B \, g(t)$. Moreover, if $\Omega$ is bounded and $p(\cdot)$ is locally log-Hölder continuous with $p_{-} > 1$, then the growth envelope of the variable Lebesgue space is
$$
\envg(L_{p(\cdot)}(\Omega)) = \left( t^{-1/p_{-}}, p_{-} \right),
$$
see Corollary \ref{cor ind leb}. For the variable Lorentz spaces when additionally $1 < q \leq \infty$, we obtain in Corollary \ref{cor ind lor},
$$
\envg(L_{p(\cdot),q}(\Omega)) = \left( t^{-1/p_{-}}, q \right).
$$
All in all, it will turn out, that the unboundedness is determined by $p_{-}$, which ``extends'' our observation from the mixed Lebesgue and Lorentz spaces in a natural way: the ``minimal'' integrability is the crucial one.

In \cite{Kempka2014}, Kempka and Vybíral defined for exponent functions $p(\cdot)$ and $q(\cdot)$, the space $L_{p(\cdot),q(\cdot)}$. It would be a natural conjecture, that this space has a growth envelope function of the form $\egv{L_{p(\cdot),q(\cdot)}}(t) = \sup\{ \|\chi_{A}\|_{L_{p(\cdot),q(\cdot)}}^{-1} : \mbox{measure of $A$ is equal to $t$} \}$. However, this space is technically so complicated that we have to postpone an answer to this question.

The paper is organized as follows. In Section~\ref{envelopes} we recall the concept of the growth envelopes, collect some of its properties and recall classical examples.

In Sections~\ref{mixed-Leb} and \ref{mixed-Lor} we concentrate on the mixed Lebesgue and mixed Lorentz spaces, respectively, and determine their growth envelopes.

We will consider the variable Lebesgue spaces in Section~\ref{var-Leb}.

In Section~\ref{var-Lor} we will prove similar theorems for the variable Lorentz space $L_{p(\cdot),q}$. Finally, in Section~\ref{appli}, we present some applications of our new results.

\section{Growth Envelope}
\label{envelopes}

First, we need the concept of the rearrangement function. Let $(\Omega,\cA,\mu)$ be a totally $\sigma$-finite measure space. For simplicity we shall restrict ourselves to the setting $\Omega\subseteq\R^d$ in what follows, where $\mu$ stands for the Lebesgue measure.
For a measurable function $\Omega \to \C$, its distribution function $\mu_{f} : [0,\infty) \to [0,\infty]$ is defined as
$$
\mu_{f}(s) := \mu( \{ |f| > s \} ), \qquad s \geq 0.
$$
It is easy to see, that $\mu_{f}$ is non-negative and non-increasing. The non-increasing rearrangement function $f^{\ast} : [0,\infty) \to [0,\infty]$ is defined by
\begin{equation*}\label{f*}
f^{\ast}(t) := \inf\{ s > 0 : \mu_{f}(s) \leq t \}, \qquad t \geq 0.
\end{equation*}
As usual, the convention $\inf \emptyset =\infty$ is assumed. In particular, for a measurable set $A$, we have
\begin{equation}\label{rearr char}
\chi_{A}^*(t) = \chi_{[0,\mu(A))}(t), \qquad t \geq 0.
\end{equation}
This is a well-known concept, see for instance \cite{Bennett1988} for further details on this subject. Clearly, $f^{\ast}(0)=\|f\|_{\infty}$ and, if $f$ is compactly supported, i.e., $\mu(\mbox{supp} f) < \infty$, then $f^{\ast}(t) = 0$ for all $t > \mu(\mbox{supp}f)$.

\begin{dfn}
\label{defi_eg}
Let $\Omega\subseteq\R^d$ and $X=X(\Omega)$ be some quasi-normed function space on $\Omega$. The {\em growth envelope function} $\egX : (0,\infty) \rightarrow [0,\infty]$ of $\ X\ $
is defined by
\[
\egX(t) := \sup_{\|f\mid X\|\leq
1} f^*(t), \qquad 0 < t < \infty.
\]
\end{dfn}

The growth envelope function was first introduced and  studied in \cite[Chapter~2]{Triebel2001} and \cite{Haroske2002}; see also \cite{Haroske2007}. 

Strictly speaking, we obtain equivalence classes of growth envelope functions
when working with equivalent quasi-norms in $ X $: if $\|\cdot\|_{1} \sim \|\cdot\|_{2}$, then $\egv{X,\|\cdot\|_{2}}(\cdot) \sim \egv{X,\|\cdot\|_{1}}(\cdot)$. 
But we do not want to
distinguish between representative and equivalence classes in what
follows and thus stick with the notation introduced above.

The following result can be found in \cite[Proposition~3.4]{Haroske2007}.

\begin{prop} \label{properties:EG}
\begin{enumerate}
\item
Let $X_i=X_i(\R^d)$, $i\in\{1,2\}$, be
function spaces on $\R^d$. Then $X_1\hookrightarrow X_2$ implies
that there exists a positive constant $C$ such that
\begin{equation*}
\egv{X_1}(t)\ \leq \ C\ \egv{X_2}(t), \qquad 0 < t < \infty,
\label{eg-XX}
\end{equation*}
where $C$ can be chosen as $\|\mathrm{id}: X_{1} \to X_{2}\|$.
\item
The embedding $X(\R^d)\hookrightarrow L_{\infty}(\R^d)$ holds if,
and only if, $\egX$ is bounded.
\end{enumerate}
\end{prop}

\begin{remark}\label{fund}
Let $X=X(\R^d)$ be a rearrangement-invariant Banach function space, $t>0$, and $ A_t
\subset \R^d $ with $\mu(A_t) = t$. Then the {\em fundamental function}
$\varphi_{X}$ of $X$ is defined by $\ 
\varphi_{X}(t) = \big\| \chi_{A_t} \big\|_{X}$. In
\cite[Sect.~2.3]{Haroske2007} it was shown that in this case
\[\egX(t) \sim  \frac{1}{\varphi_{X}(t)} = \|\chi_{A_t}\|_{X}^{-1}, \qquad t>0.\]
\end{remark}
Below we shall extend this type of characterization beyond rearrangement-invariant Banach function spaces.

Usually the envelope function $\egX(\cdot)$ is equipped with some additional fine index $\uGindex$ that contains further information: Assume that $X \not\hookrightarrow L_{\infty}$. Let $\egX(\cdot)$ be the growth envelope function of $X$ and suppose that $\egX(\cdot)$ is continuously differentiable. Then the infimum of the numbers $0 < v \leq \infty$, for which
$$
\left( - \int_{0}^{\varepsilon} \left( \frac{f^{*}(t)}{\egX(t)} \right)^v \, \frac{(\egX)'(t)}{\egX(t)} \dint t \right)^{1/v} \leq c \, \|f\|_{X}
$$
for some $c > 0$ and all functions $f \in X$ (with the usual modification if $v=\infty$) is the \emph{additional index} of $X$ and is denoted by $\uGindex$. The pair $\envg(X) := ( \egX, \uGindex)$ is called the \emph{growth envelope of the function space $X$}.


\begin{prop}\label{index est}
Let $X_{i}$ ($i=1,2$) be function spaces on $\Omega$ and suppose that $X_{1} \hookrightarrow X_{2}$. If for some $\varepsilon > 0$, $\egv{X_{1}}(t) \sim \egv{X_{2}}(t)$ for $t \in (0,\varepsilon)$, then $\uGindexv{X_{1}} \leq \uGindexv{X_{2}}$. 
\end{prop}

Now we recall some classical examples. 
For $0 < p \leq \infty$, we define the classical Lebesgue space $L_{p}$ and $\|\cdot\|_{p}$ as usual. 
If $0 < p < \infty$, then
\begin{equation}\label{formula}
\|f\|_{p} = \|f^{\ast}\|_{L_{p}\left((0,\infty)\right)} = \left( \int_{0}^{\infty} (f^{\ast})^p(t) \dint t \right)^{1/p} = \left( p \int_{0}^{\infty} s^{p-1} \, \mu_{f}(s) \dint s \right)^{1/p}.
\end{equation}
For $0 < p,q \leq \infty$, the classical \emph{Lorentz spaces} contain all measurable functions for which the quasi-norm
$$
\|f\|_{L_{p,q}} := \begin{cases} \left( \int_{0}^{\infty} \left(t^{1/p} \, f^{\ast}(t) \right)^q \, \frac{\dint t}{t} \right)^{1/q}, & \text{if $0< q < \infty$,} \\
\sup_{t \in (0,\infty)} t^{1/p} f^{\ast}(t), & \text{if $q=\infty$}
\end{cases}
$$
is finite. In particular, if $q = p$, then by \eqref{formula}, $L_{p,p}=L_{p}$, so the Lorentz spaces generalize the classical Lebesgue spaces. 
For further details we refer to \cite[Ch.~4, 
  Defs.~4.1]{Bennett1988}, for instance.

\begin{example}\label{Prop-Lpq}
Concerning growth envelopes for Lebesgue and Lorentz spaces, it was shown in \cite[Sect. 2.2]{Haroske2007}, that for $0<p<\infty$ and $0<q\leq \infty$,
\begin{equation}\label{Lp-global}
\envg(L_{p}) = \left( t^{-1/p}, p \right) \quad \mbox{and} \quad \envg(L_{p,q}) = \left( t^{-1/p}, q \right).
\end{equation}
More precisely, taking care of the (usually hidden) constants, the growth envelope function of $L_{p,q}$ (see for example in Haroske \cite{Haroske2007}) is
\begin{equation}\label{class env}
\egv{L_{p,q}}(t) = \left( \frac{q}{p} \right)^{1/q} \, t^{-1/p}, \qquad 0 < t < \mu(\Omega).
\end{equation}
\end{example}

\section{The mixed Lebesgue space}\label{mixed-Leb}

Let $d \in \N$ and $(\Omega_{i}, \mathcal{A}_{i}, \mu_{i})$ be measure spaces for $i=1,\dots,d$, and $\pv := \left(p_{1},\ldots,p_{d}\right)$ with $0 < p_{i} \leq \infty$. Consider the product space $(\Omega,\cF,\mu)$, where $\Omega = \prod_{i=1}^{d} \Omega_{i}$, $\cA$ is generated by $\prod_{i=1}^{d} \cA_{i}$ and $\mu$ is generated by $\prod_{i=1}^{d} \mu_{i}$. A measurable function $f : \Omega \to \R$ belongs to the mixed $L_{\pv}$ space if
\begin{eqnarray*}
\left\|f\right\|_{\pv} &:=& \left\|f\right\|_{(p_{1},\ldots,p_{d})} := \left\| \ldots \left\|f\right\|_{L_{p_{1}}(\dint x_{1})} \ldots \right\|_{L_{p_{d}}(\dint x_{d})} \\
&=& \left( \int_{\Omega_{d}} \ldots \left( \int_{\Omega_{1}} \left| f(x_{1},\ldots,x_{d})\right|^{p_{1}} \, \dint x_{1} \right)^{p_{2}/p_{1}} \ldots \dint x_{d} \right)^{1/p_{d}} < \infty
\end{eqnarray*}
with the usual modification if $p_{j} = \infty$ for some $j \in \left\{1,\ldots,d\right\}$. In general, the mixed Lebesgue space will be denoted by $L_{\pv}$, but if the domain is important, for example, if it is bounded, we write $L_{\pv}(\Omega)$.

If for some $0 < p \leq \infty$, $\pv = \left(p, \ldots, p\right)$, we get back the classical Lebesgue space, i.e., $L_{\pv} = L_{p}$ in this case. This means that the mixed Lebesgue spaces are generalizations of the classical Lebesgue spaces. Throughout the paper, $0<\pv \leq \infty$ will mean that the coordinates of $\pv$ satisfy the previous condition, e.g. for all $i=1,\ldots,d$, $0 < p_{i} \leq \infty$. When $\mu(\Omega) < \infty$, Benedek and Panzone \cite{Benedek1961} showed, that if $\pv \leq \qv$, then $L_{\qv}(\Omega) \hookrightarrow L_{\pv}(\Omega)$, which implies
\begin{equation}\label{embedding mixed leb}
L_{\pv}(\Omega) \hookrightarrow L_{\min\{ p_{1}, \ldots, p_{d} \}}(\Omega).
\end{equation}

In the next theorem we show, that the space $L_{\min\{p_{1}, \ldots, p_{d} \}}(\Omega)$ is indeed the smallest classical Lebesgue space, which contains the mixed Lebesgue space $L_{\pv}(\Omega)$. 
\begin{thm}\label{narrowest embedding}
Let $\mu(\Omega) < \infty$, $0 < \pv \leq \infty$ with $0 < \min\{p_{1}, \ldots, p_{d} \} < \infty$. Then for all $\varepsilon > 0$, we have
$$
L_{\pv}(\Omega) \not\hookrightarrow L_{\min\{p_{1}, \ldots, p_{d} \}+ \varepsilon}(\Omega).
$$
\end{thm}
\begin{proof}
For the sake of simplicity, suppose that $[0,1]^d \subset \Omega$ and $p_{l} := \min\{p_{1}, \ldots, p_{d} \}$. Then $0 < p_{l} < \infty$. We assume that that $\varepsilon > 0$ is sufficiently small, that is, $\varepsilon$ satisfies $p_{l} + \varepsilon < p_j$ for all $p_{j}$ for which $p_{j} > p_{l}$. Now, for $p_{j} < \infty$, let us consider the numbers
\begin{eqnarray}
0 < &\alpha_{j}& < \frac{1}{p_{j}}, \quad \mbox{if $p_{j} > p_{l};$} \label{alpha1}\\
\frac{1}{p_{l}+\varepsilon} \leq &\alpha_{j}& < \frac{1}{p_{l}}, \quad \mbox{if $p_{j} = p_{l}$} \label{alpha2}.
\end{eqnarray}

For those $j = 1, \ldots, d$, for which $p_{j} < \infty$, we consider the functions $f_{j}(x_{j}) := x_{j}^{-1/\alpha_{j}}$ and if $p_{j}=\infty$, we put $f_{j}(x_{j}) := 1$, where $x_{j} \in (0,1]$. Let us define the function
$$
f(x) := \prod_{j=1}^{d} f_{j}(x_{j}) = \prod_{j \in \{1, \ldots, d \}, \, p_{j} < \infty} \frac{1}{x_{j}^{\alpha_{j}}}, \qquad  x=(x_{1},\ldots,x_{d}) \in (0,1]^d, 
$$
and let $f(x) := 0$, if $x \in \Omega \setminus (0,1]^d$.

By \eqref{alpha1} and \eqref{alpha2}, if $p_{j} < \infty$, then $\alpha_{j} p_{j} < 1$ and therefore
$$
\|f\|_{\pv} = \prod_{j=1}^{d} \|f_{j}\|_{p_{j}} = \prod_{j \in \{1, \ldots, d \}, \, p_{j} < \infty} \left( \int_{0}^{1} \frac{1}{x_{j}^{\alpha_{j} p_{j}}} \, \dint x_{j} \right)^{1/p_{j}} < \infty,
$$
that is, $f \in L_{\pv}(\Omega)$.

By the construction (see \eqref{alpha1}), if $p_{l} < p_{j} < \infty$, then $\varepsilon>0$ was chosen such that $p_{l}+\varepsilon < p_{j}$, that is $\alpha_{j}(p_l + \varepsilon) \leq \alpha_{j}p_{j} < 1$, and if $p_{j} = p_{l}$, then $\alpha_{j} (p_{l} + \varepsilon) \geq 1$. Hence
\begin{eqnarray*}
\|f\|_{p_{l}+\varepsilon} &=& \prod_{j=1}^{d} \|f_{j}\|_{p_{l}+\varepsilon} = \prod_{j\in \{1, \dots, d\}, \ p_{l} < p_{j} < \infty} \|f_{j}\|_{p_{l}+ \varepsilon} \, \prod_{j\in\{1,\dots, d\}, \ p_{j}=p_{l}}\|f_{j}\|_{p_{l}+\varepsilon} \\
&=& \prod_{j\in \{1, \dots, d\}, \atop p_{l} < p_{j} < \infty} \left( \int_{0}^{1} \frac{1}{x_{j}^{\alpha_{j}(p_{l}+\varepsilon)}}  \, \dint x_{j} \right)^{1/(p_{l}+\varepsilon)} \, \prod_{j\in \{1, \dots, d\}, \atop p_{j} = p_{l}} \left( \int_{0}^{1} \frac{1}{x_{j}^{\alpha_{j}(p_{l}+\varepsilon)}} \, \dint x_{j} \right)^{1/(p_{l}+\varepsilon)}.
\end{eqnarray*}
In the first product, $\alpha_{j}(p_{l}+\varepsilon)< 1$, therefore the first term is finite. But, the second term is infinite since $\alpha_{j}(p_{l}+\varepsilon) \geq 1$, which means that  $f \notin L_{p_{l}+\varepsilon}(\Omega)$ implying $L_{\pv}(\Omega) \not\hookrightarrow L_{\min\{p_{1},\ldots,p_{d} \}+\varepsilon}(\Omega)$.
\end{proof}

From \eqref{embedding mixed leb} and Theorem \ref{narrowest embedding}, we obtained, that
$$
L_{\pv}(\Omega) \not\hookrightarrow L_{\infty}(\Omega)
$$
if and only if $\min\{p_{1}, \ldots, p_{d} \} < \infty$.

From \eqref{embedding mixed leb} and Proposition~\ref{properties:EG}, it follows that, if $\mu(\Omega) < \infty$ and $\min\{p_{1}, \ldots, p_{d} \} < \infty$, then
$$
\egv{L_{\pv}(\Omega)}(t) \leq c \, \egv{L_{\min\{ p_{1}, \ldots, p_{d} \}}(\Omega)}(t) = c \, t^{-1/\min\{ p_{1}, \ldots, p_{d} \}}, \quad 0 < t < \mu(\Omega),
$$
where $c$ is the embedding constant, hence $c \leq \|\chi_{\Omega}\|_{\overrightarrow{q}}$ with $1/\min\{p_{1}, \ldots, p_{d} \} = 1/p_{i} + 1/q_{i}$ $(i=1,\ldots,d)$. For the lower estimate, we need the following lemma.

\begin{lem}\label{lemma char}
Let $A_{i} \subset \cA_{i}$ with $\mu_{i}(A_{i}) < \infty$ ($i=1,\ldots,d$) and consider their Cartesian product $A:=A_{1} \times \cdots \times A_{d}$. Then
$$
\|\chi_A\|_{\pv} = \mu_{1}(A_{1})^{1/p_{1}} \, \mu_{2}(A_{2})^{1/p_{2}} \, \ldots \, \mu_{d}(A_{d})^{1/p_d}.
$$
\end{lem}
\begin{proof}
Indeed,
\begin{eqnarray*}
\|\chi_A\|_{\pv} &=& \left( \int_{\Omega_{d}} \ldots \left( \int_{\Omega_{1}} \left| \chi_{A_{1}}(x_{1}) \, \ldots \, \chi_{A_{d}}(x_{d})\right|^{p_{1}} \, \dint x_{1} \right)^{p_{2}/p_{1}} \ldots \dint x_{d} \right)^{1/p_{d}} \\
&=& \left( \int_{A_{d}} \ldots \left( \int_{A_{1}} 1 \, \dint x_{1} \right)^{p_{2}/p_{1}} \ldots \dint x_{d} \right)^{1/p_{d}} \\
&=& \mu_{1}(A_{1})^{1/p_{1}} \left( \int_{A_{d}} \ldots \left( \int_{A_{2}} 1 \, \dint x_{2} \right)^{p_{3}/p_{2}} \ldots \dint x_{d} \right)^{1/p_{d}} = \ldots \\
&=& \mu_{1}(A_{1})^{1/p_{1}} \, \mu_{2}(A_{2})^{1/p_{2}} \, \ldots \, \mu_{d}(A_{d})^{1/p_d},
\end{eqnarray*}
which proves the lemma.
\end{proof}
We have the following lower estimate for $\egv{L_{\pv}}$.
\begin{prop}\label{prop}
If $\min\{p_{1}, \ldots, p_{d} \} < \infty$, then
$$
\egv{L_{\pv}}(t) \geq t^{-1/\min\{p_{1}, \ldots, p_{d} \}}, \qquad 0 < t < \mu(\Omega).
$$
\end{prop}
\begin{proof}
Suppose that $p_{k} = \min\{ p_{1}, \ldots, p_{d} \}$ and for a fixed $t > 0$, let $s > t$. Consider the following function
$$
f_{s}(x) := s^{-1/p_{k}} \, \chi_{A_{1}^{(1)}}(x_{1}) \ldots \chi_{A_{1}^{(k-1)}}(x_{k-1}) \, \chi_{A_{s}^{(k)}}(x_{k}) \, \chi_{A_{1}^{(k+1)}}(x_{k+1}) \ldots \chi_{A_{1}^{(d)}}(x_{d}),
$$
where $x=(x_{1},\ldots,x_{d}) \in \Omega$, $A_{1}^{(i)} \in \cA_{i}$ with $\mu_{i}(A_{1}^{(i)})=1$ $(i \in \{1, \ldots,k-1,k+1,\ldots,d \})$ and $A_{s}^{(k)} \subset \cA_{k}$ with $\mu_{k}(A_{s}^{(k)}) = s$. Then by Lemma \ref{lemma char},
\begin{eqnarray*}
  \|f_{s}\|_{\pv} &=& s^{-1/p_{k}}
\left\|\chi_{A_1^{(1)} \times \cdots \times A_1^{(k-1)} \times A_{s} \times A_1^{(k+1)} \times \cdots \times A_1^{(d)}} \right\|_{\pv} \\
&=& s^{-1/p_{k}} \, s^{1/p_{k}}\\
&=& 1
\end{eqnarray*}
and by \eqref{rearr char}
$$
\egv{L_{\pv}}(t) \geq \sup_{s > t} f_{s}^{\ast}(t) = \sup_{s > t} s^{-1/p_{k}} = t^{-1/p_{k}} = t^{-1/\min\{ p_{1}, \ldots, p_{d} \}},
$$
which finishes the proof.
\end{proof}

In conclusion, for the growth envelope function $\egv{L_{\pv}(\Omega)}$ we obtain the following result.
\begin{thm}\label{cor mixed leb}
If $\mu(\Omega)<\infty$ and $\min\{p_{1}, \ldots, p_{d} \} < \infty$, then
$$
t^{-1/\min\{ p_{1}, \ldots, p_{d} \}} \leq \egv{L_{\pv}(\Omega)}(t) \leq \|\chi_{\Omega}\|_{\overrightarrow{q}} \, t^{-1/\min\{ p_{1}, \ldots, p_{d} \}}, \quad 0 < t < \mu(\Omega),
$$
that is,
$$
\egv{L_{\pv}(\Omega)}(t) \sim t^{-1/\min\{ p_{1}, \ldots, p_{d} \}}, \qquad 0 < t < \mu(\Omega).
$$
\end{thm}

If $0 < p < \infty$, $\pv = (p,\ldots,p)$ and $\mu(\Omega) < \infty$, then $L_{(p,\ldots,p)} = L_{p}$ and we recover the result for classical Lebesgue spaces, cf. \eqref{Lp-global},
$$
\egv{L_{(p,\ldots,p)}(\Omega)}(t) \sim t^{-1/p} = \egv{L_{p}(\Omega)}(t), \qquad 0 < t < \mu(\Omega).
$$

Now let us study the additional index $\uGindexv{L_{\pv}}$ of the mixed Lebesgue space $L_{\pv}$. To this, we need the following lemma. The proof can be found in \cite{Haroske2007}.

\begin{lem}
Let $s>0$ and $A_{s} := (-s/2,s/2)$. Then for all $0 < r < \infty$ and $\gamma \in \R$, if
$$
f_{s,\gamma}(x) := |x|^{-1/r} \, (1 + |\log |x||)^{-\gamma} \, \chi_{A_{s}}(x), \qquad x \in \R,
$$
then
$$
f_{s,\gamma}^{*}(t) = t^{-1/r} \, \left( 1 + |\log t| \right)^{-\gamma} \, \chi_{[0,s)}(t), \qquad t > 0.
$$
\end{lem}

\begin{thm}\label{theo-ind-Lpv}
If $\mu(\Omega)<\infty$ and $\min\{p_{1}, \ldots, p_{d} \} < \infty$, then $\uGindexv{L_{\pv}(\Omega)} = \min\{ p_{1}, \ldots, p_{d} \}$.
\end{thm}
\begin{proof}
Using Theorem \ref{cor mixed leb}, we obtain that there exists $\varepsilon > 0$, such that
$$
\egv{L_{\pv}(\Omega)}(t) \sim t^{-1/\min\{ p_{1}, \ldots, p_{d} \}} = \egv{L_{\min\{ p_{1}, \ldots, p_{d} \}}(\Omega)}(t), \qquad 0 < t < \varepsilon.
$$
By \eqref{embedding mixed leb}, Proposition~\ref{index est} and \eqref{Lp-global}, we have that $\uGindexv{L_{\pv}(\Omega)} \leq \uGindexv{L_{\min\{ p_{1}, \ldots, p_{d} \}}(\Omega)} = \min\{ p_{1}, \ldots, p_{d} \}$.

On the other hand, since $\egv{L_{\pv}(\Omega)}(t) = t^{-1/\min\{p_{1}, \ldots, p_{d} \}}$ is continuously diffe\-rentiable and
$$
\frac{\left(\egv{L_{\pv}(\Omega)}\right)'(t)}{\egv{L_{\pv}(\Omega)}(t)}\dint t \sim - \frac{\dint t}{t},
$$
we look for the smallest $0 < v \leq \infty$, such that there exists some $c>0$ with
\begin{equation}\label{ddh-1}
\left(\int_{0}^{\varepsilon}  \left[t^{1/\min\{p_{1}, \ldots, p_{d} \}} \, f^{*}(t) \right]^{v} \, \frac{\dint t}{t}  \right)^{1/v} \leq c \, \|f\|_{\pv},
\end{equation}
for all $f\in L_{\pv}(\Omega)$. Let us denote $p_l := \min\{p_{1}, \ldots, p_d \}$. We will show, that if $v < p_l$, then \eqref{ddh-1} does not hold. First, let us choose $\gamma \in \R$, such that $1/p_{l} < \gamma < 1/v$ and define the function
$$
f(x) := \chi_{A_{1}^{(1)}}(x_{1}) \ldots \chi_{A_{1}^{(l-1)}}(x_{l-1}) \, g_{s,\gamma}(x_{l}) \, \chi_{A_{1}^{(l+1)}}(x_{l+1}) \ldots \chi_{A_{1}^{(d)}}(x_d),
$$
where $x=(x_{1},\ldots,x_{d})$, $A_{1}^{(i)} \in \Omega_{i}$ with $\mu_{i}(A_1^{(i)}) = 1$ $(l \neq i = 1,\ldots,d)$ and
$$
g_{s,\gamma}(x_{l}) := |x_{l}|^{-1/p_{l}} \left( 1 + |\log |x_{l}||  \right)^{-\gamma} \, \chi_{A_{s}^{(l)}}(x_{l}),
$$
where $0 < s < 1$ and $A_{s}^{(l)} \in \Omega_{l}$ with $\mu_{l}(A_{s}^{(l)}) = s$. Then $\|f\|_{\pv} = \|g_{s,\gamma}\|_{p_l}$ and since $p_{l} \gamma > 1$,
$$
\|g_{s,\gamma}\|_{p_{l}}^{p_{l}} = \int_{0}^{s} \frac{1}{x_{l} (1-\log x_{l})^{p_{l}\gamma}} \, \dint x_{l} < \infty,
$$
that is, $\|f\|_{\pv} < \infty$. It is easy to see that
$$
f^{*}(t) = g_{s,\gamma}^{*}(t) = t^{-1/p_{l}} \left( 1 + |\log t| \right)^{-\gamma} \, \chi_{[0,s)}(t), \qquad t > 0.
$$
Hence,
$$
\left(\int_{0}^{s}  \left[t^{1/p_{l}} \, f^{*}(t) \right]^{v} \, \frac{\dint t}{t}  \right)^{1/v} = \left( \int_{0}^{s} \frac{1}{(1+|\log t|)^{\gamma v}} \frac{\dint t}{t} \right)^{1/v} = \infty,
$$
because of $\gamma v < 1$. Altogether, we get that $\uGindexv{L_{\pv}(\Omega)} = \min\{p_{1}, \ldots, p_{d} \}$.
\end{proof}

In conclusion, we obtain the following result for growth envelopes in mixed Lebesgue spaces.

\begin{cor}\label{cor env mixleb}
If $\mu(\Omega) < \infty$ and $\min\{p_{1}, \ldots, p_{d} \} < \infty$, then
$$
\envg(L_{\pv}(\Omega)) = \left( t^{-1/\min{ \{ p_{1}, \ldots, p_{d} \}}}, \min{ \{ p_{1}, \ldots, p_{d} \}} \right).
$$
\end{cor}

\section{The mixed Lorentz space}\label{mixed-Lor}

It is known, that $L_{\infty,\infty} = L_{\infty}$, and for all $0 < q < \infty$, the space $L_{\infty,q}$ contains the zero function only. Therefore, if $p=\infty$, then it is supposed that $q = \infty$, too. Moreover, for $0 < p < \infty$, it follows from (see \cite[Prop. 1.4.9.]{Grafakos2008} and \cite{Bennett1988,Triebel2015}), that the quasi-norm of the classical Lorentz space can be written as
\begin{equation}\label{lorentz-norm}
\|f\|_{L_{p,q}} = \begin{cases} p^{1/q} \left( \int_{0}^{\infty} u^q \|\chi_{\{ |f| > u \}}\|_{p}^{q} \, \frac{\dint u}{u} \right)^{1/q}, & \text{if $0 < q < \infty$},\\
\sup_{u \in (0,\infty)} u \, \|\chi_{\{|f| > u \}}\|_{p}, & \text{if $q = \infty$}.
\end{cases}
\end{equation}
Therefore the quasi-norm
\begin{equation}\label{lorentz-equiv}
\|f\|_{\tilde{L}_{p,q}} := \begin{cases} \left( \int_{0}^{\infty} u^q \|\chi_{\{ |f| > u \}}\|_{p}^{q} \, \frac{\dint u}{u} \right)^{1/q}, & \text{if $0 < q < \infty$},\\
\sup_{u \in (0,\infty)} u \, \|\chi_{\{|f| > u \}}\|_{p}, & \text{if $q = \infty$}
\end{cases}
\end{equation}
is equivalent with the previous one. This approach allows for a generalization to mixed Lorentz spaces and later on to variable Lorentz spaces.

For a vector $0 < \pv \leq \infty$ and for a number $0 < q \leq \infty$, the mixed Lorentz space $L_{\pv,q}$ contains all measurable functions for which the quasi-norm
$$
\|f\|_{L_{\pv,q}} := \begin{cases} \left( \int_{0}^{\infty} u^q \left\| \chi_{\{ |f| > u \}} \right\|_{\pv}^q \, \frac{\dint u}{u} \right)^{1/q}, & \text{if $0 < q < \infty$}\\
\sup_{u \in (0,\infty)} u \, \| \chi_{\{ |f| > u \}} \|_{\pv}, & \text{if $q = \infty$}
\end{cases}
$$
is finite. If it does not cause misunderstanding, the mixed Lorentz space is denoted by $L_{\pv,q}$, but if the domain is important, for example, if it is bounded, then we denote this by $L_{\pv,q}(\Omega)$. If $\pv = (p,\ldots,p)$, where $0 < p < \infty$, then for all $0 < q < \infty$, by $\|f\|_{L_{(p,\ldots,p)}} = \|f\|_{p}$, we see from the definition of $\|\cdot\|_{L_{\pv,q}}$, that
\begin{equation}\label{equiv norm2}
\|f\|_{L_{(p,\ldots,p),q}} = \|f\|_{\widetilde{L}_{p,q}} = p^{-1/q} \, \|f\|_{L_{p,q}}
\end{equation}
and in case $q = \infty$ it holds $\|f\|_{L_{(p,\ldots,p),\infty}} = \|f\|_{\widetilde{L}_{p,\infty}} = \|f\|_{L_{p,\infty}}$.

\begin{lem}\label{lemma emb lorentz}
If $0 < q_{1} \leq q_{2} \leq \infty$, then for all $0 < \pv \leq \infty$ we have the embedding
$$
L_{\pv, q_{1}} \hookrightarrow L_{\pv,q_{2}}.
$$
\end{lem}
\begin{proof}
Let us start with the case $q_{2} = \infty$. then for all $s > 0$,
\begin{eqnarray*}
q_{1}^{1/q_{1}} \, \|f\|_{L_{\pv,q_{1}}} \geq q_{1}^{1/q_{1}} \, \left( \int_{0}^{s} u^{q_{1}} \, \|\chi_{\{|f| > u \}}\|_{\pv}^{q_{1}} \, \frac{\dint u}{u} \right)^{1/q_{1}} \geq \|\chi_{\{|f| > s \}}\|_{\pv} \cdot s,  
\end{eqnarray*}
which implies
$$
\|f\|_{L_{\pv,\infty}} = \sup_{s > 0} s \, \|\chi_{\{ |f| > s \}}\|_{\pv} \leq q_{1}^{1/q_{1}} \, \|f\|_{L_{\pv,q_{1}}}. 
$$
And if $0 < q_{1} < q_{2} < \infty$, then by the previous inequality,
\begin{eqnarray*}
\|f\|_{L_{\pv,q_{2}}} &=& \left( u^{q_{2}} \|\chi_{\{ |f| > u \}}\|_{\pv}^{q_{2}} \, \frac{\dint u}{u} \right)^{1/q_{2}} \\
&\leq& \left( \sup_{u > 0} u \|\chi_{ |f| > u }\|_{\pv} \right)^{\frac{q_{2}-q_{1}}{q_{2}}} \left( \int_{0}^{\infty} u^{q_{1}} \|\chi_{\{ |f| > u \}}\|_{\pv}^{q_{1}} \, \frac{\dint u}{u} \right)^{1/q_{1} \cdot q_{1}/q_{2}} \\
&\leq& \|f\|_{L_{\pv,q_{1}}}^{\frac{q_{2}-q_{1}}{q_{2}}} \, \|f\|_{L_{\pv,q_{1}}}^{\frac{q_{1}}{q_{2}}},
\end{eqnarray*}
which means that $L_{\pv,q_{1}} \hookrightarrow L_{\pv,q_{2}}$ and the proof is complete.
\end{proof}

If $\mu\left(\Omega\right) < \infty$ and $\overrightarrow{r} \leq \pv$, then $L_{\pv}(\Omega) \hookrightarrow L_{\overrightarrow{r}}(\Omega)$ (see Benedek and Panzone \cite{Benedek1961}). Thus, for all $u > 0$, $\left\|\chi_{ \{ |f| > u \} }\right\|_{\overrightarrow{r}} \leq c \, \left\|\chi_{ \{ |f| > u \} }\right\|_{\pv}$ and therefore for all $0 < q < \infty$,
\begin{eqnarray*}
 \|f\|_{L_{\overrightarrow{r},q}} = \left( \int_{0}^{\infty} u^q \left\| \chi_{\{ |f| > u \}} \right\|_{\overrightarrow{r}}^q \, \frac{\dint u}{u} \right)^{1/q} \leq \, c  \left( \int_{0}^{\infty} u^q \left\| \chi_{\{ |f| > u \}} \right\|_{\pv}^q \, \frac{\dint u}{u} \right)^{1/q}  = c \,  \|f\|_{L_{\pv,q}}.
\end{eqnarray*}
The case $q = \infty$ can be handled similarly. This means, that if $\mu(\Omega) < \infty$ and $\overrightarrow{r} \leq \pv$, then for all $0 < q \leq \infty$, the embedding $L_{\pv,q}(\Omega) \hookrightarrow L_{\overrightarrow{r},q}(\Omega)$ holds. As a special case, if $\min\{p_{1}, \ldots, p_{d} \} < \infty$, we have that $(\min\{p_{1},\ldots,p_{d} \}, \ldots, \min\{p_{1}, \ldots, p_{d}\}) \leq \pv$ and therefore (see \eqref{equiv norm2}),
\begin{eqnarray*}
\left(\min\{p_{1},\ldots,p_{d} \}\right)^{-1/q} \, \|f\|_{L_{\min\{p_{1},\ldots,p_{d} \},q}} \leq c \, \|f\|_{L_{\pv,q}},
\end{eqnarray*}
where $0 < q < \infty$. Thus, $\|f\|_{L_{\min\{p_{1},\ldots,p_{d} \},q}} \leq c \, \left(\min\{p_{1},\ldots,p_{d} \}\right)^{1/q} \, \|f\|_{L_{\pv,q}}$. Hence, by \eqref{class env}, we get that for all $0 < t < \mu(\Omega)$,
\begin{eqnarray}\label{lorentz upper}
\egv{L_{\pv,q}(\Omega)}(t) \leq c \, \left(\min\{p_{1},\ldots,p_{d} \}\right)^{1/q} \, \egv{L_{\min\{p_{1},\ldots,p_{d} \},q}(\Omega)}(t) = c \, q^{1/q} \, t^{-1/\min\{p_{1}, \ldots, p_{d}\}},
\end{eqnarray}
where $c$ can be estimated by $\|\chi_{\Omega}\|_{\overrightarrow{s}}$, where $1/\min\{p_{1}, \ldots, p_{d}\} = 1/p_{i} + 1/s_{i}$ $(i=1,\ldots,d)$. Similarly, if $q = \infty$, then $\egv{L_{\pv,\infty}(\Omega)}(t) \leq \|\chi_{\Omega}\|_{\overrightarrow{s}} \, t^{-1/\min\{p_{1}, \ldots, p_{d}\}}$. The following lower estimate holds for $\egv{L_{\pv,q}}$.
\begin{prop}
If $0 < \min\{p_{1}, \ldots, p_{d} \} < \infty$ and $0 < q \leq \infty$, then
$$
\egv{L_{\pv,q}}(t) \geq q^{1/q} \, t^{-1/\min\{ p_{1}, \ldots, p_{d} \}}, \qquad t > 0 
$$
where in case of $q = \infty$, $\infty^{1/\infty} := 1$.
\end{prop}
\begin{proof}
Again, suppose that $p_{k} = \min\{ p_{1}, \ldots, p_{d} \}$ and for a fixed $t > 0$, let $s > t$. Let us consider the function
$$
f_{s}(x) := s^{-1/p_{k}} \, \chi_{A_{1}^{(1)}}(x_{1}) \ldots \chi_{A_{1}^{(k-1)}}(x_{k-1}) \, \chi_{A_{s}^{(k)}}(x_{k}) \, \chi_{A_{1}^{(k+1)}}(x_{k+1}) \ldots \chi_{A_{1}^{(d)}}(x_{d}),
$$
where $x=(x_{1},\ldots,x_{d}) \in \Omega$, $A_{1}^{(i)} \in \cA_{i}$ with $\mu_{i}(A_{1}^{(i)})=1$ $(i \in \{1, \ldots,k-1,k+1,\ldots,d \})$ and $A_{s}^{(k)} \subset \cA_{k}$ with $\mu_{k}(A_{s}^{(k)}) = s$. Then by Lemma \ref{lemma char},
\begin{eqnarray*}
\|q^{1/q} f_{s}\|_{L_{\pv,q}} &=& q^{1/q} \, s^{-1/p_{k}} \, \left\|\chi_{A_{1}^{(1)} \times \cdots \times A_{1}^{(k-1)} \times A_{s}^{(k)} \times A_{1}^{(k+1)} \times \cdots \times A_{1}^{(d)} } \right\|_{L_{\pv,q}} \\
&=& q^{1/q} \, s^{-1/p_{k}} \, \left( \int_{0}^{1} u^q \left\| \chi_{A_{1}^{(1)} \times \cdots \times A_{1}^{(k-1)} \times A_{s}^{(k)} \times A_{1}^{(k+1)} \times \cdots \times A_{1}^{(d)} } \right\|_{\pv}^{q} \, \frac{\dint u}{u} \right)^{1/q} \\
&=& q^{1/q} \cdot s^{-1/p_{k}} \cdot s^{1/p_{k}} \cdot q^{-1/q} \\
&=& 1
\end{eqnarray*}
thus,
\begin{equation}\label{lorentz lower}
\egv{L_{\pv,q}}(t) \geq \sup_{s > t} \left(q^{1/q} f_{s}\right)^{\ast}(t) = q^{1/q} \, \sup_{s > t} s^{-1/p_{k}} = q^{1/q} \cdot t^{-1/p_{k}} = q^{1/q} \cdot t^{-1/\min\{ p_{1}, \ldots, p_{d} \}},
\end{equation}
which proves the proposition.
\end{proof}

In terms of the growth envelope function our previous results yield the following.
\begin{thm}\label{func equiv lorentz}
If $\mu(\Omega) < \infty$, $0 < \min\{p_{1}, \ldots, p_{d} \} < \infty$ and $0 < q \leq \infty$, then
$$
q^{1/q} \, t^{-1/\min\{ p_{1}, \ldots, p_{d} \}} \leq \egv{L_{\pv,q}(\Omega)}(t) \leq \|\chi_{\Omega}\|_{\overrightarrow{s}} \, q^{1/q} \, t^{-1/\min\{ p_{1}, \ldots, p_{d} \}}, \qquad 0 < t < \mu(\Omega),
$$
where $1/\min\{p_{1}, \ldots, p_{d}\} = 1/p_{i} + 1/s_{i}$, i.e.,
$$
\egv{L_{\pv,q}(\Omega)}(t) \sim t^{-1/\min\{ p_{1}, \ldots, p_{d} \}}, \qquad 0 < t < \mu(\Omega). 
$$
\end{thm}

If $\pv = (p,\ldots,p)$ with $0 < p < \infty$, then
$$
\egv{L_{(p,\ldots,p),q}(\Omega)}(t) \sim t^{-1/p},
$$
which is equivalent with the classical result \eqref{class env}.

Concerning the additional index $\uGindexv{L_{\pv,q}}$ of the mixed Lorentz space $L_{\pv,q}$ we can state the following.

\begin{thm}\label{theo-ind-Lpqv}
If $\mu(\Omega)<\infty$, $0 < \min\{p_{1}, \ldots, p_{d} \} < \infty$ and $0 < q \leq \infty$, then
$$
\uGindexv{L_{\pv,q}(\Omega)} = q.
$$
\end{thm}
\begin{proof}
Using Theorem \ref{func equiv lorentz}, we obtain that there exists $\varepsilon > 0$, such that
$$
\egv{L_{\pv,q}(\Omega)}(t) \sim t^{-1/\min\{ p_{1}, \ldots, p_{d} \}} \sim \egv{L_{\min\{ p_{1}, \ldots, p_{d} \},q}(\Omega)}(t), \quad 0 < t < \varepsilon.
$$
Using the embedding $L_{\pv,q}(\Omega) \hookrightarrow L_{\min{\{p_{1},\ldots, p_{d} \},q}}(\Omega)$, Proposition~\ref{index est} and \eqref{Lp-global}, we have that
$$
\uGindexv{L_{\pv,q}(\Omega)} \leq \uGindexv{L_{\min\{ p_{1}, \ldots, p_{d} \},q}(\Omega)} = q.
$$

We put again $p_l := \min\{p_{1}, \ldots, p_d \}$ and suppose that $q < \infty$. We will show, that if $v < q$, then the inequality
\begin{equation}\label{ddh-3}
\left(\int_{0}^{\varepsilon}  \left[t^{1/\min\{p_{1}, \ldots, p_{d} \}} \, f^{*}(t) \right]^{v} \, \frac{\dint t}{t}  \right)^{1/v} \leq c \, \|f\|_{L_{\pv,q}}
\end{equation}
does not hold for all $f \in L_{\pv,q}$. Let $\gamma \in \R$ again, such that $1/q < \gamma < 1/v$ and consider the same function as in the proof of Theorem \ref{theo-ind-Lpv}:
$$
f(x) := \chi_{A_{1}^{(1)}}(x_{1}) \ldots \chi_{A_{1}^{(l-1)}}(x_{l-1}) \, g_{s,\gamma}(x_{l}) \, \chi_{A_{1}^{(l+1)}}(x_{l+1}) \ldots \chi_{A_{1}^{(d)}}(x_d),
$$
where $x=(x_{1},\ldots,x_{d})$, $A_{1}^{(i)} \in \Omega_{i}$ with $\mu_{i}(A_1^{(i)}) = 1$ $(l \neq i = 1,\ldots,d)$ and
$$
g_{s,\gamma}(x_{l}) := |x_{l}|^{-1/p_{l}} \left( 1 + |\log |x_{l}||  \right)^{-\gamma} \, \chi_{A_{s}^{(l)}}(x_{l}),
$$
where $0 < s < 1$ and $A_{s}^{(l)} \in \Omega_{l}$ with $\mu_{l}(A_{s}^{(l)}) = s$. It is easy to see, that
$$
\{ |f| > t \} = A_{1}^{(1)} \times \cdots \times A_{1}^{(l-1)} \times \{ |g_{s,\gamma}| > t \} \times A_{1}^{(l)} \times \cdots \times A_{1}^{(d)}.
$$
Therefore,
$$
\left\| \chi_{\{ |f| > t \}} \right\|_{\pv} = \left\| \chi_{\{ |g_{s,\gamma}| > t \}} \right\|_{p_{l}},
$$
that is, by \eqref{lorentz-norm},
\begin{eqnarray*}
\|f\|_{L_{\pv,q}} &=& \left( \int_{0}^{\infty} t^q \left\| \chi_{\{ |g_{s,\gamma}| > t \}} \right\|_{p_{l}}^{q} \, \frac{\dint t}{t} \right)^{1/q} =p_{l}^{-1/q} \left( \int_{0}^{\infty} \left[ t^{1/p_{l}} g_{s,\gamma}^{*}(t) \right]^{q} \, \frac{\dint t}{t} \right)^{1/q} \\
&=& p_{l}^{-1/q} \left( \int_{0}^{s} \frac{1}{(1+|\log t|)^{\gamma q}} \frac{\dint t}{t} \right)^{1/q} < \infty,
\end{eqnarray*}
because $\gamma q > 1$. This means that $f \in L_{\pv,q}(\Omega)$. At the same time, in the proof of Theorem \ref{theo-ind-Lpv}, we have seen that the left-hand side of \eqref{ddh-3} is not finite, since $\gamma v > 1$. Hence, it follows that $v \geq q$, which implies $\uGindexv{L_{\pv,q}(\Omega)} \geq q$.
Together with the first part of the proof, we have that $\uGindexv{L_{\pv,q}(\Omega)} = q$.

Let $q = \infty$. Now, for an arbitrary $0 <v < \infty$, let us choose a number $\gamma > 0$, such that $\gamma \, v < 1$. Then by the same extremal function $f$, we have again, that $\left\| \chi_{\{ |f| > t \}} \right\|_{\pv} = \left\| \chi_{\{ |g_{s,\gamma}| > t \}} \right\|_{p_{l}}$ and therefore by the definition of the $\|\cdot\|_{L_{p_l,\infty}}$ quasi-norm and \eqref{lorentz-norm}, we obtain that
$$
\|f\|_{L_{\pv,\infty}} = \sup_{t > 0} t \, \left\| \chi_{\{ |f| > t \}} \right\|_{\pv} =  \sup_{t > 0} t \, \left\| \chi_{\{ |g_{s,\gamma}| > t \}} \right\|_{p_{l}} = \sup_{t > 0} t^{1/p_{l}} \, g_{s,\gamma}^*(t) \leq 1,
$$
that is, $f \in L_{\pv,\infty}$. Since $\gamma > 0$, recall the proof of Theorem \ref{theo-ind-Lpv}. We have seen, that the integral on the left-hand side of \eqref{ddh-3} is infinite and the proof is complete. 
\end{proof}

Altogether, in terms of growth envelopes for mixed Lorentz spaces we have obtained the following.

\begin{cor}\label{cor env mixlor}
If $\mu(\Omega) < \infty$, $0 < \min\{p_{1}, \ldots, p_{d} \} < \infty$ and $0 < q \leq \infty$, then
$$
\envg(L_{\pv,q}(\Omega)) = \left( t^{-1/\min{ \{ p_{1}, \ldots, p_{d} \}}}, q \right).
$$
\end{cor}

\section{The variable Lebesgue space}\label{var-Leb}

We can generalize the classical Lebesgue space $L_{p}$ in another way. In this case, the exponent will not be a vector, but a function of $x$. Let $\Omega \subseteq \R^d$, $p(\cdot) : \Omega \to (0,\infty)$ be a measurable function and denote
$$
p_{-} := \einf_{x \in \Omega} p(x), \qquad p_{+} := \esup_{x \in \Omega} p(x).
$$
Similarly, for a measurable set $A$,
$$
p_{-}(A) := \einf_{x \in A} p(x), \qquad p_{+}(A) := \esup_{x \in A} p(x).
$$
If $p_{-} > 0$, then we say that $p(\cdot)$ is an exponent function. Moreover, the set of all exponent functions is denoted by $\cP$. For $p(\cdot) \in \cP$ and for a measurable function $f$, the $p(\cdot)$-modular is defined by
$$
\varrho_{p(\cdot)}(f) := \int_{\Omega} |f(x)|^{p(x)} \, \dint x,
$$
where $\dint x$ denotes the Lebesgue measure. A measurable function $f$ belongs to the variable Lebesgue space $L_{p(\cdot)}$, if for some $\lambda > 0$, $\varrho_{p(\cdot)}(f/\lambda) < \infty$. Endowing this space with the quasi-norm
$$
\|f\|_{p(\cdot)} := \inf \left\{ \lambda > 0 : \varrho_{p(\cdot)}\left( \frac{f}{\lambda} \right) \leq 1 \right\},
$$
we get a quasi-normed space $(L_{p(\cdot)},\|\cdot\|_{p(\cdot)})$. In general, we denote the variable Lebesgue space by $L_{p(\cdot)}$, except the domain is important. In particular, if $\mu(\Omega) < \infty$, then the variable Lebesgue space on $\Omega$ is denoted by $L_{p(\cdot)}(\Omega)$. If the function $p(\cdot) = p$ is constant, we get back the classical Lebesgue space $L_{p}$. If $\mu(\Omega) < \infty$ and $r(\cdot) \leq p(\cdot)$ pointwise, then (see e.g., Diening \cite{Diening2011})
\begin{equation}\label{embedding var lebesgue}
L_{p(\cdot)}(\Omega) \hookrightarrow L_{r(\cdot)}(\Omega).
\end{equation}

We have the following inequalities. If $p_{+} < \infty$, then for any $|\lambda| \leq 1$ and $|\widetilde{\lambda}| > 1$
\begin{eqnarray}
|\lambda|^{p_{+}(\mathrm{supp} (f))} \varrho_{p(\cdot)}(f) &\leq& \varrho_{p(\cdot)}(\lambda f) \leq |\lambda|^{p_{-}(\mathrm{supp} (f))} \varrho_{p(\cdot)}(f), \label{modular ineq 1} \\
|\widetilde{\lambda}|^{p_{-}(\mathrm{supp} (f))} \varrho_{p(\cdot)}(f) &\leq& \varrho_{p(\cdot)}(\widetilde{\lambda} f) \leq |\widetilde{\lambda}|^{p_{+}(\mathrm{supp} (f))} \varrho_{p(\cdot)}(f), \label{modular ineq 2}
\end{eqnarray}
where the set $\mathrm{supp} (f)$ denotes the support of $f$. From this, it follows, that for all $f \in L_{p(\cdot)}$, the map $\alpha \mapsto \varrho_{p(\cdot)}(\alpha f)$ is increasing. Indeed, suppose, that $\alpha_{1} < \alpha_{2}$. Then $\alpha_{2} / \alpha_{1} > 1$, and therefore $(\alpha_{2}/\alpha_{1})^{p_{-}} > 1$, too. Thus
$$
\varrho_{p(\cdot)}\left(\alpha_{2} f \right) = \varrho_{p(\cdot)}\left( \frac{\alpha_{2}}{\alpha_{1}} \alpha_{1} f \right) \geq \left(\frac{\alpha_{2}}{\alpha_{1}}\right)^{p_{-}} \varrho_{p(\cdot)}(\alpha_{1} f) > \varrho_{p(\cdot)}\left(\alpha_{1} f \right).
$$
From this, we get as well, that for all $f \in L_{p(\cdot)}$, the function $\lambda \mapsto \varrho_{p(\cdot)}(f / \lambda)$ is non-increasing (moreover, decreasing). 

Besides that, the $\|\cdot\|_{p(\cdot)}$-quasi-norm of the function $f$ can be estimated by (see \cite{Cruz-Uribe2014})
\begin{eqnarray}
\varrho_{p(\cdot)}(f)^{1/p_{-}(\mathrm{supp} (f))} &\leq& \|f\|_{p(\cdot)} \leq \varrho_{p(\cdot)}(f)^{1/p_{+}(\mathrm{supp} (f))}, \quad \varrho_{p(\cdot)}(f) \leq 1, \label{norm estimate 1}\\
\varrho_{p(\cdot)}(f)^{1/p_{+}(\mathrm{supp} (f))} &\leq& \|f\|_{p(\cdot)} \leq \varrho_{p(\cdot)}(f)^{1/p_{-}(\mathrm{supp} (f))}, \quad \varrho_{p(\cdot)}(f) > 1 \label{norm estimate 2}.
\end{eqnarray}
Since $\varrho_{p(\cdot)} (\chi_{A}) = \mu(A)$, using \eqref{modular ineq 1} and \eqref{modular ineq 2} for a characteristic function $\chi_{A}$, we get
\begin{eqnarray*}
|\lambda|^{p_{+}(A)} \mu(A) &\leq& \varrho_{p(\cdot)}(\lambda \chi_{A}) \leq |\lambda|^{p_{-}(A)} \mu(A), \qquad  |\lambda| \leq 1, \label{char modular 1}\\
|\lambda|^{p_{-}(A)} \mu(A) &\leq& \varrho_{p(\cdot)}(\lambda \chi_{A}) \leq |\lambda|^{p_{+}(A)} \mu(A), \qquad |\lambda| > 1  \label{char modular 2}
\end{eqnarray*}
and the quasi-norm of a characteristic function $\chi_{A}$ can be estimated (see \eqref{norm estimate 1} and \eqref{norm estimate 2}) by
\begin{eqnarray}
\mu(A)^{1/p_{-}(A)} &\leq& \|\chi_{A}\|_{p(\cdot)} \leq \mu(A)^{1/p_{+}(A)}, \qquad \mu(A) \leq 1, \label{char norm 1}\\
\mu(A)^{1/p_{+}(A)} &\leq& \|\chi_{A}\|_{p(\cdot)} \leq \mu(A)^{1/p_{-}(A)}, \qquad \mu(A) > 1 \label{char norm 2}.
\end{eqnarray}

The proof of the following theorem for $p(\cdot) \in \cP$ with $p_{-} \geq 1$ can be found in \cite{Diening2011}. If $p_{-} < 1$ the proof is similar using inequality \eqref{modular ineq 1}.
\begin{thm}[Norm-modular unit ball property]\label{norm-modular for Lp}
    Let $p(\cdot): \R^d \to (0,\infty)$.
    \begin{enumerate}
    \item Then for all $f \in L_{p(\cdot)}$, $\|f\|_{p(\cdot)} \leq 1$ and $\varrho_{p(\cdot)}(f) \leq 1$ are equivalent.
    \item If $p_{+} < \infty$, then also $\|f\|_{p(\cdot)} < 1$ and $\varrho_{p(\cdot)}(f) < 1$ are equivalent, as are $\|f\|_{p(\cdot)} = 1$ and $\varrho_{p(\cdot)}(f) = 1$.
    \end{enumerate} 
\end{thm}
\begin{proof}
To see 1., if $\varrho_{p(\cdot)}(f) \leq 1$, then by the definition of the quasi-norm, $\|f\|_{p(\cdot)} \leq 1$. If $\|f\|_{p(\cdot)} \leq 1$, then by monotonicity, for all $\lambda > 1$, $\varrho_{p(\cdot)}(f/\lambda) \leq 1$. We have by the left-continuity of the map $\lambda \mapsto \varrho_{p(\cdot)}(\lambda f)$, that
$$
\varrho_{p(\cdot)}(f) = \lim_{\lambda \downarrow 1}\varrho\left(\frac{f}{\lambda}\right) \leq 1,
$$
so $\varrho_{p(\cdot)}(f) \leq 1$.

Now let us see 2. If $p_{+} < \infty$, then the function $\lambda \mapsto \varrho_{p(\cdot)}(\lambda \, f)$ is continuous. If $\varrho_{p(\cdot)}\left(f\right) < 1$, then there exists $\gamma > 1$, such that $\varrho_{p(\cdot)}(\gamma \, f) < 1$. Thus, $\|\gamma \, f\|_{p(\cdot)} \leq 1$, that is, $\left\|f\right\|_{p(\cdot)} \leq 1/\gamma < 1$. For the reverse statement, suppose that $\left\|f\right\|_{p(\cdot)} < 1$. Then there exists $\lambda_{0} < 1$, such that $\varrho_{p(\cdot)}(f/\lambda_{0}) \leq 1$. Hence by \eqref{modular ineq 1},
$$
\varrho_{p(\cdot)}\left(f\right) = \varrho_{p(\cdot)}\left(\lambda_{0} \, \frac{f}{\lambda_{0}}\right) \leq \lambda_{0}^{p_{-}} \,  \varrho_{p(\cdot)}\left(\frac{f}{\lambda_{0}}\right) < 1.
$$

The equivalence of $\varrho_{p(\cdot)}(f) = 1$ and $\|f\|_{p(\cdot)} = 1$ follows from the previous cases.
\end{proof}
The following lemma can be proved easily.

\begin{lem}\label{theorems for Lp}
If $p(\cdot) \in \cP$. Then the following holds:
\begin{enumerate}
\item $\|f\|_{p(\cdot)} = \left\||f|\right\|_{p(\cdot)}$;
\item if $f \in L_{p(\cdot)}$, $g$ is measurable and $|g| \leq |f|$ almost everywhere, then $g \in L_{p(\cdot)}$ and $\|g\|_{p(\cdot)} \leq \|f\|_{p(\cdot)}$;
\end{enumerate}
\end{lem}
We will also need the result, that if the sequence of functions $(f_{n})_{n}$ tends to $f$ in the $\|\cdot\|_{p(\cdot)}$-norm, i.e., $\|f_{n}-f\|_{p(\cdot)} \to 0$, then the sequence of the norms $(\|f_{n}\|_{p(\cdot)})_{n}$ tends to $\|f\|_{p(\cdot)}$. This is very easy, if we have the triangle inequality. Indeed, in this case, 
$$
0 \leq \left| \|f_{n}\|_{p(\cdot)} - \|f\|_{p(\cdot)} \right| \leq \|f_{n}-f\|_{p(\cdot)} \to 0 \quad \left(n \to \infty \right).
$$ 
But, if $p_{-} < 1$, the space $L_{p(\cdot)}$ is not a Banach space, just a quasi-Banach space, and the triangle inequality does not hold. We circumvent this problem by using the following lemma. The proof can be found in \cite{Cruz-Uribe2014}.

\begin{lem}\label{lem uribe}
Let $p(\cdot) : \R^d \to (0,\infty)$ and $p_{+} < \infty$. If $p_{-} \leq 1$, then for every $f,g \in L_{p(\cdot)}$,
$$
\|f+g\|_{p(\cdot)}^{p_{-}} \leq \|f\|_{p(\cdot)}^{p_{-}} + \|g\|_{p(\cdot)}^{p_{-}}.
$$
\end{lem}
If $p_{-} > 1$, then Lemma \ref{lem uribe} is not true. But in this case, the triangle inequality holds. Using these observations, we get the following result.

\begin{lem}\label{conv norms lem}
Let $p(\cdot): \R^d \to (0,\infty)$ with $p_{+} < \infty$ and $f_{k}, f \in L_{p(\cdot)}$ $(k \in \N)$. If $\lim_{k \to \infty} f_{k} = f$ in the $L_{p(\cdot)}$-norm, then $\lim_{k \to \infty}\|f_{k}\|_{p(\cdot)} = \|f\|_{p(\cdot)}$. 
\end{lem}
\begin{proof}
If $p_{-} \geq 1$, then by the triangle inequality,
$$
0 \leq \left| \|f_{k}\|_{p(\cdot)} - \|f\|_{p(\cdot)} \right| \leq \|f-g\|_{p(\cdot)} \to 0 \quad \left(k \to \infty \right),
$$
so we have $\lim_{k \to \infty} \|f_{k}\|_{p(\cdot)} = \|f\|_{p(\cdot)}$.

If $p_- < 1$, then by Lemma \ref{lem uribe} we have
$$
\|f_{k}\|_{p(\cdot)}^{p_{-}} = \|f_{k} - f + f\|_{p(\cdot)}^{p_{-}} \leq \|f_{k} - f\|_{p(\cdot)}^{p_{-}} + \|f\|_{p(\cdot)}^{p_{-}}
$$
and
$$
\|f\|_{p(\cdot)}^{p_{-}} = \|f - f_{k} + f_{k}\|_{p(\cdot)}^{p_{-}} \leq \|f - f_{k}\|_{p(\cdot)}^{p_{-}} + \|f_{k}\|_{p(\cdot)}^{p_{-}},
$$
thus,
$$
0 \leq \left| \|f_{k}\|_{p(\cdot)}^{p_{-}} - \|f\|_{p(\cdot)}^{p_{-}} \right| \leq \|f-f_{k}\|_{p(\cdot)}^{p_{-}} \to 0 \quad \left(k \to \infty \right),
$$
which means that $\lim_{k \to \infty} \|f_{k}\|_{p(\cdot)}^{p_{-}} = \|f\|_{p(\cdot)}^{p_{-}}$, that is, $\lim_{k \to \infty} \|f_{k}\|_{p(\cdot)} = \|f\|_{p(\cdot)}$.
\end{proof}

Now, for fixed $t > 0$, let us consider the sets $A_{t} \subset A_{s}$ where $s > t$, $\mu(A_{t}) = t$, $\mu(A_{s})=s$, and  $\chi_{A_{t}}$ and $\chi_{A_{s}}$ denote their characteristic functions. We may suppose that $s \leq t+1$. Then $\mu(A_{s} \setminus A_{t}) \leq 1$ and by \eqref{char norm 1},
\begin{eqnarray*}
\|\chi_{A_{s}} - \chi_{A_{t}}\|_{p(\cdot)} = \|\chi_{A_{s} \setminus A_{t}}\|_{p(\cdot)} \leq \mu(A_{s} \setminus A_{t})^{1/p_{+}} = (s-t)^{1/p_{+}} \to 0 \quad \left( s \downarrow t\right),
\end{eqnarray*}
that is, $\chi_{A_{s}} \to \chi_{A_{t}}$ in the $L_{p(\cdot)}$-norm. By Lemma \ref{conv norms lem}, $\|\chi_{A_{s}}\|_{p(\cdot)} \to \|\chi_{A_{t}}\|_{p(\cdot)}$. Moreover, by Lemma \ref{theorems for Lp}, $\|\chi_{A_{s}}\|_{p(\cdot)} \searrow \|\chi_{A_{t}}\|_{p(\cdot)}$ as $s \downarrow t$, that is,
$\inf_{s > t, A_{t} \subset A_{s}}\|\chi_{A_{s}}\|_{p(\cdot)} = \lim_{s \downarrow t} \|\chi_{A_{s}}\|_{p(\cdot)} = \|\chi_{A_{t}}\|_{p(\cdot)}$ and therefore
\begin{equation}\label{sup rec ineq}
\sup_{s > t, A_{t} \subset A_{s}} \|\chi_{A_{s}}\|_{p(\cdot)}^{-1} = \|\chi_{A_{t}}\|_{p(\cdot)}^{-1}.
\end{equation}

After these preparations we now study growth envelopes of variable Lebesgue spaces. We proceed as follows: We obtain the lower estimate of the growth envelope function of the space $L_{p(\cdot)}$ under some mild condition on the exponent function $p(\cdot)$, namely that the exponent function $p(\cdot)$ is bounded. For the upper estimate, we need the condition, that $p(\cdot)$ is constant $p_{-}$ on a (small) set. Assuming that the exponent function $p(\cdot)$ is additionally locally log-H\"older continuous and $p_{-}$ is attained, we show that the growth envelope function is actually equivalent to $t^{-1/p_{-}}$ near to the origin. Moreover, in case $\Omega$ is bounded, then is proved, that the additional index of the function space $L_{p(\cdot)}(\Omega)$ is $p_{-}$.

\subsection{Lower estimate for $\egv{L_{p(\cdot)}}$}

We recall the following very simple result which follows immediately from the definition of $f^*$. If $A \subset \R^d$ is measurable, then
\begin{equation}\label{char-csillag}
\chi_{A}^{*}(t) = \chi_{[0,\mu(A))}(t), \qquad t \geq 0.
\end{equation}

\begin{prop}\label{thm lower leb}
Let $p(\cdot) \in \cP$ and $p_{+} < \infty$. Then
$$
\egv{L_{p(\cdot)}}(t) \geq \sup \left\{ \|\chi_{A_{t}}\|_{p(\cdot)}^{-1} : \mu(A_{t}) = t  \right\}, \qquad t > 0.
$$
\end{prop}
\begin{proof}
Let $t > 0$. For a fixed number $s> t$ let us choose a set $A_{s} \subset \R^d$ with $\mu(A_{s}) = s$ and consider the functions $\varphi_{s,A_{s}} := \|\chi_{A_{s}}\|_{p(\cdot)}^{-1} \chi_{A_{s}}$. First, we see that $\|\varphi_{s,A_{s}}\|_{p(\cdot)} = 1$ and using \eqref{char-csillag}, we conclude that $\varphi_{s,A_{s}}^{\ast}(t) = \|\chi_{A_{s}}\|_{p(\cdot)}^{-1}$ for $0 \leq t < s$. If we consider only the functions $\varphi_{s,A_{s}}$, we get the following lower estimate,
\begin{eqnarray*}
\egv{L_{p(\cdot)}}(t) &=& \sup_{\|f\|_{p(\cdot)} \leq 1} f^{\ast}(t) \geq \sup_{s > t, \mu(A_{s})=s} \varphi_{s,A_{s}}^{\ast}(t) \\
&=& \sup_{s > t} \|\chi_{A_{s}}\|_{p(\cdot)}^{-1} \geq \sup_{s > t, A_{t} \subset A_{s}} \|\chi_{A_{s}}\|_{p(\cdot)}^{-1}.
\end{eqnarray*}
We have seen in \eqref{sup rec ineq}, that this supremum is $\|\chi_{A_{t}}\|_{p(\cdot)}^{-1}$, i.e.,
$$
\egv{L_{p(\cdot)}}(t) \geq \|\chi_{A_{t}}\|_{p(\cdot)}^{-1},
$$
where the set $A_{t}$ was an arbitrary set with measure $t$. Thus
$$
\egv{L_{p(\cdot)}}(t) \geq  \sup \left\{ \|\chi_{A_{t}}\|_{p(\cdot)}^{-1} : \mu(A_{t}) = t \right\}
$$
and the proof is complete.
\end{proof}

\subsection{Upper estimate for $\egv{L_{p(\cdot)}}$}

For the upper estimate we need to assume more conditions on the exponent function $p(\cdot)$.

\begin{thm}\label{thm upper old}
Let $p(\cdot) \in \cP$ with $p_{+} < \infty$ and suppose that there exists a set $A_{t_{0}}$, with $\mu(A_{t_{0}})=t_{0}$, such that $p(x) = p_{-}$ for all $x \in A_{t_{0}}$. Then
$$
\egv{L_{p(\cdot)}}(t) \leq \sup \left\{ \|\chi_{A_{t}}\|_{p(\cdot)}^{-1} : \mu(A_{t}) = t \right\}, \qquad 0 < t < \min\{1, t_{0} \}.
$$
\end{thm}
\begin{proof}
Let $0 < t < \min\{1, t_{0} \}$ be fixed and let us denote
$$
\alpha := \sup \left\{ \|\chi_{A_{t}}\|_{p(\cdot)}^{-1} : \mu(A_{t}) = t  \right\}.
$$
Then our claim $\egv{L_{p(\cdot)}}(t) = \sup_{\|f\|_{p(\cdot)}\leq 1} f^{\ast}(t) \leq \alpha$ means that for every $f \in L_{p(\cdot)}$ with $\|f\|_{p(\cdot)} \leq 1$, $\mu\left( \left\{ |f| > \alpha \right\} \right) \leq t$. We prove it by contradiction. Assume on the contrary, that there exists a function $f \in L_{p(\cdot)}$, $\|f\|_{p(\cdot)} \leq 1$ such that $\mu\left( \left\{ |f| > \alpha \right\} \right) > t$. We can suppose w.l.o.g. that $f \in L_{p(\cdot)}$ such that
\begin{equation}\label{we can suppose}
t < \mu\left( \left\{ |f| > \alpha \right\} \right) < 1.
\end{equation}
It is easy to see, that
\begin{equation}\label{pontonkenti becsles}
|f| \geq \alpha \chi_{ \left\{ |f| > \alpha \right\}}.
\end{equation}

Since $\varrho_{p(\cdot)}(\chi_{A_{t}}) = \mu(A_{t}) = t < 1$, by \eqref{char norm 1},
$$
\|\chi_{A_{t}}\|_{p(\cdot)} \geq \mu(A_{t})^{1/p_{-}(A_{t})} = t^{1/p_{-}(A_{t})} \geq t^{1/p_{-}}.
$$
From this, we have that
\begin{equation}\label{kisebb}
\alpha = \sup \left\{ \|\chi_{A_{t}}\|_{p(\cdot)}^{-1} :  \mu(A_{t}) = t \right\} \leq t^{-1/p_{-}}.
\end{equation}
By our general assumption there exists a set $A_{t_{0}}$ with measure $t_{0}$, such that for all $x \in A_{t_{0}}$, $p(x) = p_{-}$. It can be assumed that $t_{0} \leq 1$. By \eqref{char norm 1}, for this set $A_{t_{0}}$ we compute
\begin{eqnarray*}
\|\chi_{A_{t_{0}}}\|_{p(\cdot)} \leq \mu\left(A_{t_{0}}\right)^{1/p_{+}(A_{t_{0}})} = \mu\left(A_{t_{0}}\right)^{1/p_{-}(A_{t_{0}})} = \mu(A_{t_{0}})^{1/p_{-}} = t_{0}^{1/p_{-}},
\end{eqnarray*}
so $\|\chi_{A_{t_{0}}}\|_{p(\cdot)}^{-1} \geq t_{0}^{-1/p_{-}}$. It is clear that for all $t < t_{0}$, and a set $A_{t} \subset A_{t_{0}}$ we have for all $x \in A_{t}$, $p(x) = p_{-}$, too. Therefore, $\|\chi_{A_{t}}\|_{p(\cdot)}^{-1} \geq t^{-1/p_{-}}$. Thus, for all $t < t_{0}$,
\begin{equation}\label{nagyobb}
\alpha = \sup \left\{ \|\chi_{A_{t}}\|_{p(\cdot)}^{-1} : \mu(A_{t})=t \right\} \geq \sup \left\{ \|\chi_{A_{t}}\|_{p(\cdot)}^{-1} :  \mu(A_{t})=t, \, p|_{A_{t}} = p_{-} \right\} \geq t^{-1/p_{-}}.
\end{equation}
By \eqref{kisebb} and \eqref{nagyobb}, we obtain for all $t < t_{0}$,
\begin{equation}\label{sup egyenlo}
\alpha = \sup \left\{ \|\chi_{A_{t}}\|_{p(\cdot)}^{-1} : \mu(A_{t})=t \right\} = t^{-1/p_{-}}.
\end{equation}

Using \eqref{pontonkenti becsles}, Theorem \ref{theorems for Lp}, \eqref{char norm 1} (with the condition \eqref{we can suppose}) and \eqref{sup egyenlo}, it follows
\begin{eqnarray*}
1 &\geq& \|f\|_{p(\cdot)} \geq \alpha \| \chi_{\{|f|>\alpha\}}\|_{p(\cdot)} \geq \alpha \mu(\{ |f| > \alpha \})^{1/p_{-}}\\
&>& \alpha \, t^{1/p_{-}} = t^{-1/p_{-}} \cdot t^{1/p_{-}} = 1
\end{eqnarray*}
so we have that $1 > 1$, which is a contradiction. Hence,
$$
\egv{L_{p(\cdot)}}(t) \leq \sup\left\{ \|\chi_{A_{t}}\|_{p(\cdot)}^{-1} : \mu(A_{t}) = t \right\},
$$
which proves the theorem.
\end{proof}

By Proposition \ref{thm lower leb} and Theorem \ref{thm upper old}, the following corollary is obtained.

\begin{cor}\label{cor old}
Let $p(\cdot) \in \cP$ with $p_{+} < \infty$ and suppose that there exists a set $A_{t_{0}}$, with $\mu(A_{t_{0}})=t_{0}$, such that $p(x) = p_{-}$ for all $x \in A_{t_{0}}$. Then
$$
\egv{L_{p(\cdot)}}(t) = \sup \left\{ \|\chi_{A_{t}}\|_{p(\cdot)}^{-1} : \mu(A_{t}) = t \right\}, \qquad 0 < t < \min\{1, t_{0} \}.
$$
\end{cor}

\begin{remark}
If $p(\cdot) = p$, then for any set $A$, $p|_{A} = p =p_-$ and for all sets $A_{t}$ with measure $t$, $\|A_{t}\|_{p} = \mu(A_{t})^{1/p} = t^{1/p}$, hence in this case
$$
\egv{L_{p}}(t)  = t^{-1/p}, \qquad 0 < t < 1,
$$
that is, we get back the classical result.
\end{remark}

\begin{remark}\label{remark-fund-leb}
Obviously the space $L_{p(\cdot)}$ is not rearrangement-invariant for arbitrary $p(\cdot) \in \cP$ satisfying the assumptions of Corollary \ref{cor old}. So this can be seen now as the extension of our result connecting the growth envelope function $\egX$ and fundamental function $\varphi_X$ in rearrangement-invariant spaces, see Remark \ref{fund}, to more general spaces.
\end{remark}

\begin{remark}\label{remark upper}
The condition for the upper estimate that the exponent function $p(\cdot)$ is constant $p_{-}$ on a set, may be too strong. This condition can be omitted, if we suppose that $\mu(\Omega) < \infty$. Indeed, in this case we have the embedding $L_{p(\cdot)}(\Omega) \hookrightarrow L_{p_{-}}(\Omega)$ and therefore, see Proposition~\ref{properties:EG} and \eqref{Lp-global}, we obtain that
\begin{equation*}\label{upper}
\egv{L_{p(\cdot)}(\Omega)}(t) \leq c \, \egv{L_{p_{-}}(\Omega)}(t) = c \, t^{-1/p_{-}}.
\end{equation*}
Using this together Theorem \ref{thm lower leb}, we see that if $\mu(\Omega) < \infty$ and $p_{+} < \infty$, then
$$
\sup \left\{ \|\chi_{A_{t}}\|_{p(\cdot)}^{-1} : \mu(A_{t}) = t  \right\} \leq \egv{L_{p(\cdot)}(\Omega)}(t) \leq c \, t^{-1/p_{-}}, \qquad 0 < t < \mu(\Omega).
$$
\end{remark}

In what follows we show that if we additionally assume the exponent function $p(\cdot)$ to be locally log-Hölder continuous at a point $x_{0}$, where $p(x_{0}) = p_{-}$, then the lower estimate in Remark \ref{remark upper} can be replaced by $c \, t^{-1/p_{-}}$. The function $r(\cdot)$ is \emph{locally $\log$-Hölder continuous at the point $x_{0}$}, if there exists a constant $C_{0} > 0$, such that for all $x \in \Omega$, $|x-x_{0}| < 1/2$,
$$
|r(x)-r(x_{0})| \leq \frac{C_{0}}{- \log(|x-x_{0}|)}.
$$
We will denote this by $r(\cdot) \in LH_{0}\{ x_{0} \}$. If the previous condition holds for all $x_{0} \in \Omega$, then $r(\cdot)$ is \emph{locally log-Hölder continuous} (not only in $x_{0}$), in notation $r(\cdot) \in LH_{0}$. The ball with radius $r>0$ and center $x_{0}$ is denoted by $B_{r}(x_{0}) := \{ y \in \Omega : \|y-x_{0}\|_{2} < r \}$. The following lemma is similar as in \cite[Lemma 3.24.]{Cruz-Uribe2013} or in \cite[Lemma 4.1.6.]{Diening2011}.

\begin{lem}\label{lem log-holder}
Let $p(\cdot) \in \cP$ and suppose that $x_{0} \in \Omega$, such that $p_{-} = p(x_{0})$. Then the function $p(\cdot)$ is locally $\log$-Hölder continuous at $x_{0}$ if, and only if, there exists $C > 0$, such that for all $r > 0$,
\begin{equation}\label{equat}
\mu\left( B_{r}(x_{0}) \right)^{p_{-}(B_{r}(x_{0}))-p_{+}(B_{r}(x_{0}))} \leq C.
\end{equation}
\end{lem}
\begin{proof}
If $r \geq 1/2$, then by the positivity of the exponent $p_{+}(B_{r}(x_{0})) - p_{-}(B_{r}(x_{0}))$,
$$
\mu\left(B_{r}(x_{0})\right)^{p_{+}(B_{r}(x_{0})) - p_{-}(B_{r}(x_{0}))} \geq \mu\left(B_{\frac{1}{2}}(x_{0})\right)^{p_{+}(B_{r}(x_{0})) - p_{-}(B_{r}(x_{0}))} \geq \mu\left(B_{\frac{1}{2}}(x_{0})\right)^{p_{+} - p_{-}},
$$
that is,
$$
\mu\left( B_{r}(x_{0}) \right)^{p_{-}(B_{r}(x_{0}))-p_{+}(B_{r}(x_{0}))} \leq \mu\left(B_{\frac{1}{2}}(x_{0})\right)^{p_{-}-p_{+}} =: C.
$$
Now, suppose that $r < 1/2$. It is enough to show that for some constant $C > 0$,
$$
\log\left( \mu\left( B_{r}(x_{0}) \right)^{p_{-}(B_{r}(x_{0}))-p_{+}(B_{r}(x_{0}))} \right) \leq C.
$$
Since $p_{-}(B_{r}(x_{0})) - p_{+}(B_{r}(x_{0})) = -|p_{-}(B_{r}(x_{0})) - p_{+}(B_{r}(x_{0}))|$, the left-hand side is equal to
$$
\left|p_{-}(B_{r}(x_{0})) - p_{+}(B_{r}(x_{0})) \right| \, \log\left( \frac{1}{\mu(B_{r}(x_{0}))} \right).
$$
By our assumptions $p_{-} = p(x_{0})$, that is, for all $r > 0$, $p_{-}(B_{r}(x_{0})) = p_{-}$. Hence, for any $x \in \overline{B_{r}(x_{0})}$, $|x-x_{0}| \leq r < 1/2$. Using that $p(\cdot) \in LH_{0}\{ x_{0} \}$, we obtain
\begin{eqnarray*}
\lefteqn{ \left|p_{-}(B_{r}(x_{0})) - p_{+}(B_{r}(x_{0})) \right| \, \log\left( \frac{1}{\mu(B_{r}(x_{0}))} \right) }\\
&\leq& \frac{c_{1}}{-\log\left( r \right)} \, \log \left( \frac{1}{\mu\left( B_{r}(x_{0}) \right)} \right) = \frac{c_{1}}{\log\left(\frac{1}{r} \right)} \, \log\left( \frac{1}{c_{d} r^d} \right) \\
&=& \frac{c_{1}d \log\left( \frac{1}{r} \right)}{\log\left( \frac{1}{r} \right)} + \frac{c_{1} \log\left( \frac{1}{c_{d}} \right)}{\log\left( \frac{1}{r} \right)} \leq c_{1} d + \frac{c_{1} \log\left( \frac{1}{c_{d}} \right)}{\log(2)} =: C.
\end{eqnarray*}
Now let us see the other direction and suppose that \eqref{equat} holds. Let $x,x_{0} \in \Omega$ with $|x-x_{0}|< 1/2$. Then there exists a ball $B_{r}(x_{0})$, such that $\mu(B_{r}(x_{0})) \leq |x-x_{0}|^{d}$ and therefore $\mu(B_{r}(x_{0})) < 1/2^d < 1$. By \eqref{equat},
\begin{eqnarray*}
C &\geq& \mu\left( B_{r}(x_{0}) \right)^{p_{-}(B_{r}(x_{0})) - p_{+}(B_{r}(x_{0}))} = \left( \frac{1}{\mu(B_{r}(x_{0}))} \right)^{p_{+}(B_{r}(x_{0})) - p_{-}(B_{r}(x_{0}))} \\
&\geq& \left( \frac{1}{\mu(B_{r}(x_{0}))} \right)^{|p(x)-p(x_{0})|} = \mu(B_{r}(x_{0}))^{-|p(x)-p(x_{0})|} \geq |x-y|^{-d |p(x)-p(x_{0})|} \\
&=& \exp\left( -d \, |p(x)-p(x_{0})| \, \ln(|x-x_{0}|) \right).
\end{eqnarray*}
Taking the logarithm of both sides, after ordering we obtain
$$
|p(x)-p(x_{0})| \leq \frac{\frac{\ln(C)}{d}}{-\ln(|x-x_{0}|)} =: \frac{C_{0}}{-\ln(|x-x_{0}|)},
$$
which means, that $p(\cdot) \in LH_{0}\{x_{0}\}$ and the proof is complete.
\end{proof}
Using the previous lemma, we get the following result.
\begin{lem}\label{lem integral}
Let $p(\cdot) \in \cP$ with $p_{+}< \infty$. If there is $x_{0} \in \Omega$, such that $p(\cdot) \in LH_{0}\{ x_{0} \}$ and $p(x_{0}) = p_{-}$, then there exist $C > 0$, such that for all $r > 0$,
\begin{equation}\label{int}
\int_{B_{r}(x_{0})} \mu\left( B_{r}(x_{0}) \right)^{-p(x)/p_{-}} \, \dint x \leq C^{1/p_{-}}.
\end{equation}
\end{lem}
\begin{proof}
If $\mu(B_{r}(x_{0}))^{-1} \leq 1$, then by $p(x)/p_{-} > 1$, we have that \eqref{int} holds with $C = 1$. If $\mu(B_{r}(x_{0}))^{-1} > 1$, then by Lemma \ref{lem log-holder},
\begin{eqnarray*}
\lefteqn{\int_{B_{r}(x_{0})} \mu\left( B_{r}(x_{0}) \right)^{-p(x)/p_{-}} \, \dint x \leq \int_{B_{r}(x_{0})} \mu\left( B_{r}(x_{0}) \right)^{-p_{+}(B_{r}(x_{0}))/p_{-}} \, \dint x }\\
&=& \mu\left( B_{r}(x_{0}) \right)^{-p_{+}(B_{r}(x_{0}))/p_{-}} \, \mu\left( B_{r}(x_{0}) \right) = \mu\left( B_{r}(x_{0}) \right)^{- \frac{p_{+}(B_{r}(x_{0}))-p_{-}}{p_{-}}} \\
&=& \left[ \mu\left( B_{r}(x_{0}) \right)^{p_{-} - p_{+}(B_{r}(x_{0}))} \right]^{1/p_{-}} \\
&\leq& C^{1/p_{-}},
\end{eqnarray*} 
which proves the lemma.
\end{proof}

\begin{thm}\label{thm lower mod}
Let $p(\cdot) \in \cP$ with $p_{+} < \infty$. If there exists $x_{0} \in \Omega$, such that $p(x_{0}) = p_{-}$ and $p(\cdot) \in LH_{0}\{ x_{0} \}$, then there exists $\varepsilon > 0$, such that
$$
\egv{L_{p(\cdot)}}(t) \geq c \, t^{-1/p_{-}}, \qquad 0 < t < \varepsilon.
$$
\end{thm}
\begin{proof}
There exists $j_{0} \in \N$, such that $B_{2^{-j_{0}}}(x_{0}) \subset \Omega$. Then for $j \geq j_{0}$ we define the functions
$$
f_{j}(x) := a_{j} \, \chi_{B_{2^{-j}}(x_{0})}(x), \qquad x \in \Omega,
$$
where $a_{j} = \mu(B_{2^{-j}}(x_{0}))^{-1/p_{-}}$. Then by Lemma \ref{lem integral},
$$
\varrho_{p(\cdot)}(f_{j}) = \int_{B_{2^{-j}}(x_{0})} \mu\left( B_{2^{-j}}(x_{0}) \right)^{-p(x)/p_{-}} \, \dint x \leq C^{1/p_{-}}.
$$
If $C^{1/p_{-}} \geq 1$, then by \eqref{modular ineq 1},
$$
\varrho_{p(\cdot)}\left( \left( C^{1/p_{-}} \right)^{-1/p_{-}} \, f_{j} \right) \leq \left( C^{1/p_{-}} \right)^{-1} \varrho_{p(\cdot)}(f_{j}) \leq 1.
$$
And if $C^{1/p_{-}} < 1$, then by \eqref{modular ineq 2},
$$
\varrho_{p(\cdot)}\left( \left( C^{1/p_{-}} \right)^{-1/p_{+}} \, f_{j} \right) \leq \left( C^{1/p_{-}} \right)^{-1} \varrho_{p(\cdot)}(f_{j}) \leq 1,
$$
that is, by the norm-modular unit ball property (see Theorem \ref{norm-modular for Lp}), the $\|\cdot\|_{p(\cdot)}$-norm of the function
$$
\varphi_{j} := \min \left\{ \left( C^{1/p_{-}} \right)^{-1/p_{-}} , \left( C^{1/p_{-}} \right)^{-1/p_{+}} \right\} f_{j}
$$
is $\|\varphi_{j}\|_{p(\cdot)} \leq 1$. By \eqref{char-csillag},
$$
\varphi_{j}^{*}(t) = \min \left\{ \left( C^{1/p_{-}} \right)^{-1/p_{-}} , \left( C^{1/p_{-}} \right)^{-1/p_{+}} \right\} \, \mu\left( B_{2^{-j}}(x_{0}) \right)^{-1/p_{-}},
$$
if $0 < t < \mu(B_{2^{-j}}(x_{0}))$. We have for any $0 < h < 1$,
\begin{eqnarray*}
\egv{L_{p(\cdot)}}\left( h \, \mu\left( B_{2^{-j}}(x_{0}) \right) \right) &\geq& \varphi_{j}^{*}\left( h \, \mu\left( B_{2^{-j}}(x_{0}) \right) \right) \\
&=& \min \left\{ \left( C^{1/p_{-}} \right)^{-1/p_{-}} , \left( C^{1/p_{-}} \right)^{-1/p_{+}} \right\} \, \mu\left( B_{2^{-j}}(x_{0}) \right)^{-1/p_{-}}.
\end{eqnarray*}
Since the function $\egv{L_{p(\cdot)}}$ is decreasing (see Proposition~\ref{properties:EG}), we obtain
\begin{eqnarray*}
\egv{L_{p(\cdot)}}\left( \mu\left(B_{2^{-j}}(x_{0})\right) \right) &=& \inf \left\{ \egv{L_{p(\cdot)}}(h_{d} \, \mu\left( B_{2^{-j}}(x_{0}) \right) ) : 0 < h < 1 \right\} \\
&\geq& \min \left\{ \left( C^{1/p_{-}} \right)^{-1/p_{-}} , \left( C^{1/p_{-}} \right)^{-1/p_{+}} \right\} \, \mu\left( B_{2^{-j}}(x_{0}) \right)^{-1/p_{-}},
\end{eqnarray*}
for all $j \geq j_{0}$. This implies
$$
\egv{L_{p(\cdot)}}(t) \geq  C \, t^{-1/p_{-}}, \qquad 0  < t < \varepsilon,
$$
where $\varepsilon := \mu\left( B_{2^{-j_{0}}}(x_{0})\right)$.
\end{proof}
   
Summing up our previous results we obtain the following.

\begin{cor}\label{cor equiv lebesgue}
Let $\mu(\Omega) < \infty$, $p(\cdot) \in \cP$ with $p_{+} < \infty$. If there exists $x_{0} \in \Omega$, such that $p(x_{0}) = p_{-}$ and $p(\cdot) \in LH_{0}\{ x_{0} \}$, then there exists $\varepsilon > 0$, such that
$$
\egv{L_{p(\cdot)}(\Omega)}(t) \sim t^{-1/p_{-}}, \qquad 0 < t < \varepsilon.
$$
\end{cor}

Moreover, under the condition that $p(\cdot) \in LH_{0}$, a similar upper estimate can be reached without assuming that $\Omega$ is bounded. In order to show this, we need the following lemmas. The proofs can be found in \cite[Lemma 4.1.6.]{Diening2011} and \cite[Lemma 5.2.]{Jiao2016}.
\begin{lem}\label{lemma log-holder}
Let $p(\cdot): \R^d \to (0,\infty)$, $p_{-} > 0$ (then $1/p_{-}=(1/p)_{+} < \infty$). Then the following assertions are equivalent:
\begin{enumerate}
    \item $1/p(\cdot) \in LH_{0}$;
    \item there exists $C > 0$, such that for all cubes $Q \subset \R^d$ and $x \in Q$, $\mu(Q)^{1/p_{+}(Q) -1/p(x)} \leq C$;
    \item there exists $C > 0$, such that for all cubes $Q \subset \R^d$ and $x \in Q$, $\mu(Q)^{1/p(x) -1/p_{-}(Q)} \leq C$;
    \item there exists $C > 0$, such that for all cubes $Q \subset \R^d$, $\mu(Q)^{1/p_{+}(Q) -1/p_{-}(Q)} \leq C^2$.
\end{enumerate}
Instead of cubes, it is also possible to use balls.
\end{lem}

\begin{lem}
Let $p(\cdot): \R^d \to (0,\infty)$, $0 < p_{-} \leq p_{+} < \infty$. If $1/p(\cdot) \in LH_{0}$, then for all cubes $Q \subset \R^d$ with $\mu(Q) \leq 1$ and for all $x \in Q$,
$$
\|\chi_{Q}\|_{p(\cdot)} \sim \mu(Q)^{1/p_{-}(Q)} \sim \mu(Q)^{1/p_{+}(Q)} \sim \mu(Q)^{1/p(x)}.
$$
Again, instead of cubes, it is also possible to use balls.
\end{lem}

From the previous lemma, we get, that if $p_{+} < \infty$ and $p(\cdot)$ is locally log-Hölder continuous (note, that if $p_{+} < \infty$, then $p(\cdot) \in LH_{0}$ if, and only if, $1/p(\cdot) \in LH_{0}$; see \cite[Proposition 2.3.]{Cruz-Uribe2013}), then for all cubes (or balls) $A_{t}$, with $\mu(A_{t}) = t \leq 1$,
\begin{equation}\label{constan c2}
c_{1} \, t^{1/p_{-}(A_{t})} = c_{1} \, \mu(A_{t})^{1/p_{-}(A_{t})} \leq \|\chi_{A_{t}}\|_{p(\cdot)} \leq  c_{2} \, \mu(A_{t})^{1/p_{-}(A_{t})} = c_{2} \, t^{1/p_{-}(A_{t})}.
\end{equation}
If $p_{-}(A_{t}) = p_{-}$, then we have, $c_{1} \, t^{1/p_{-}} \leq \|\chi_{A_{t}}\|_{p(\cdot)} \leq c_{2} \, t^{1/p_{-}}$. If $p_{-}(A_{t}) > p_{-}$, then by $t \leq 1$, we obtain, that $t^{1/p_{-}} < t^{1/p_{-}(A_{t})}$ and $\|\chi_{A_{t}}\|_{p(\cdot)} > c_{1} \, t^{1/p_{-}}$. Thus
\begin{equation}\label{sup upper}
 \sup \left\{ \|\chi_{A_{t}}\|_{p(\cdot)}^{-1} : A_{t} \subset \R^d \, \mbox{cube (or ball)}, \, \mu(A_{t}) = t \right\} \leq c_{1}^{-1} \, t^{-1/p_{-}}.
\end{equation}

Modifying the proof of Theorem \ref{thm upper old}, we obtain for all $0  < t \leq 1$, the upper estimate $\egv{L_{p(\cdot)}}(t) \leq  C \, t^{-1/p_{-}}$, without the assumption that $\mu(\Omega) < \infty$.

\begin{thm}\label{thm upper better}
Let $p(\cdot) \in \cP$ with $p_{+} < \infty$, $p(\cdot) \in LH_{0}$ and suppose that there exists $x_{0} \in \Omega$, such that $p(x_{0}) = p_{-}$. Then
$$
\egv{L_{p(\cdot)}}(t) \leq C \, t^{-1/p_{-}}, \qquad 0 < t \leq 1.
$$
\end{thm}
\begin{proof}
Let $t \in \left(0, 1\right]$ be fixed and let us denote
$$
\alpha := c_{2} \, \sup \left\{ \|\chi_{A_{t}}\|_{p(\cdot)}^{-1} : A_{t} \subset \R^d \, \mbox{cube (or ball)}, \, \mu(A_{t}) = t  \right\},
$$
where the constant $c_{2}$ is equal with the constant in \eqref{constan c2}. Similarly, as in the proof of Theorem \ref{thm upper old}, we use argument by contradiction. By the definition of $\alpha$ and \eqref{constan c2},
\begin{eqnarray}
\alpha &\geq& c_{2} \, \sup \left\{ \|\chi_{A_{t}}\|_{p(\cdot)}^{-1} : A_{t} \subset \R^d \, \mbox{cube (or ball)}, \, \mu(A_{t}) = t, \, p_{-}(A_{t}) = p_{-}  \right\} \n\\
&\geq& c_{2} \, c_{2}^{-1} \, t^{-1/p_{-}}.\label{sup}
\end{eqnarray}
As in the proof of Theorem \ref{thm upper old}, using \eqref{sup}, we have that
$$
1 \geq \|f\|_{p(\cdot)} > \alpha \, t^{1/p_{-}} \geq t^{-1/p_{-}} \cdot t^{1/p_{-}} = 1,
$$
which is a contradiction. From this, by \eqref{sup upper}, we have that
$$
\egv{L_{p(\cdot)}}(t) \leq c_{2} \, \sup \left\{ \|\chi_{A_{t}}\|_{p(\cdot)}^{-1} : A_{t} \subset \R^d \, \mbox{cube (or ball)}, \, \mu(A_{t}) = t  \right\} \leq \frac{c_{2}}{c_{1}} \, t^{-1/p_{-}},
$$
which proves the theorem.
\end{proof}

Using Theorem \ref{thm lower mod} and Theorem \ref{thm upper better}, we get the following corollary. If we take the weaker condition, that the exponent function is locally $\log$-Hölder continuous and there is a point $x_{0} \in \Omega$, where $p(x_{0}) = p_{-}$ (instead of the condition that $p(\cdot)$ is constant $p_{-}$ on a set), then the equivalence $\egv{L_{p(\cdot)}}(t) \sim t^{-1/p_{-}}$ is obtained for all $0 < t < \varepsilon$.

\begin{cor}\label{cor equiv leb}
Let $p(\cdot) \in \cP$ with $p_{+} < \infty$ and $p(\cdot) \in LH_{0}$. If there exists $x_{0} \in \Omega$, such that $p(x_{0}) = p_{-}$, then there exists $\varepsilon > 0$, such that
$$
\egv{L_{p(\cdot)}}(t) \sim t^{-1/p_{-}}, \qquad 0 < t < \varepsilon.
$$
\end{cor}

\subsection{Additional index}

We study the index $\uGindexv{L_{p(\cdot)}}$ now. Recall Proposition~\ref{index est}.

\begin{thm}\label{index var leb}
Let $\mu(\Omega) < \infty$, $p(\cdot) \in \cP$ with $1 < p_{-} \leq p_{+} <\infty$. If there exists $x_{0} \in \Omega$, such that $p(x_{0}) = p_{-}$ and $p(\cdot) \in LH_{0}\{ x_{0} \}$, then $\uGindexv{L_{p(\cdot)}(\Omega)}=p_{-}$.
\end{thm}
\begin{proof}
From Corollary \ref{cor equiv lebesgue}, we obtain, that for all $0 < t < \varepsilon$, $\egv{L_{p(\cdot)}(\Omega)}(t) \sim \egv{L_{p_{-}}(\Omega)}(t)$. Since $L_{p(\cdot)}(\Omega) \hookrightarrow L_{p_{-}}(\Omega)$ (see \eqref{embedding var lebesgue}), by Proposition~\ref{index est}, we have that $\uGindexv{L_{p(\cdot)}(\Omega)} \leq \uGindexv{L_{p_{-}}(\Omega)} = p_{-}$.

On the other hand, for any $1 \leq v < p_{-}$ there exists $\alpha > 1$, such that $1 \leq v < \alpha v \leq p_{-}$. Let us choose a $j_{0} \in \N$, such that for all $j \geq j_{0}$,
$$
\frac{\mu\left( B_{2^{-j}}(x_{0}) \right)^{-1}}{j^{\alpha}} \geq 1.
$$
For $k > j_{0}$ and $j=j_{0},\ldots,k$, let us consider
$$
b_{j} := \left(\frac{\mu\left( B_{2^{-j}}(x_{0}) \right)^{-1}}{j^{\alpha}} \right)^{1/p_{-}}, \qquad s_{k}(x) := b_{k} \, \chi_{B_{2^{-k}}(x_{0})}(x) + \sum_{j=j_{0}}^{k-1} b_{j} \, \chi_{A_{j}(x_{0})}(x), 
$$
where $A_{j}(x_{0}) = B_{2^{-j}}(x_{0}) \setminus B_{2^{-(j+1)}}(x_{0})$. Since the sets $B_{2^{-k}}(x_{0})$ and $A_{j}(x_{0})$ $(j=j_{0}, \ldots, k-1)$ are pairwise disjoint,
\begin{eqnarray*}
\varrho_{p(\cdot)}(s_{k}) &=& \int_{\Omega} \left( b_{k} \, \chi_{B_{2^{-k}}(x_{0})}(x) + \sum_{j=j_{0}}^{k-1} b_{j} \, \chi_{A_{j}(x_{0})}(x) \right)^{p(x)} \, \dint x \\
&=& \int_{B_{2^{-k}}(x_{0})} b_{k}^{p(x)} \, \dint x + \sum_{j=j_{0}}^{k-1} \int_{A_{j}(x_{0})} b_{j}^{p(x)} \, \dint x \\
&\leq& \int_{B_{2^{-k}}(x_{0})} \left( \frac{\mu\left( B_{2^{-k}}(x_{0}) \right)^{-1}}{k^{\alpha}} \right)^{p(x)/p_{-}} \, \dint x \\
& & + \sum_{j=j_{0}}^{k-1} \int_{B_{2^{-j}}(x_{0})} \left( \frac{\mu\left( B_{2^{-j}}(x_{0}) \right)^{-1}}{j^{\alpha}} \right)^{p(x)/p_{-}} \, \dint x.
\end{eqnarray*}
Since $b_j \geq 1$, we get that
\begin{eqnarray*}
\varrho_{p(\cdot)}(s_{k}) &\leq& \int_{B_{2^{-k}}(x_{0})} \left( \frac{\mu\left( B_{2^{-k}}(x_{0}) \right)^{-1}}{k^{\alpha}} \right)^{p_{+}\left(B_{2^{-k}}(x_{0})\right)/p_{-}} \, \dint x \\
& & + \sum_{j=j_{0}}^{k-1} \int_{B_{2^{-j}}(x_{0})} \left( \frac{\mu\left( B_{2^{-j}}(x_{0}) \right)^{-1}}{j^{\alpha}} \right)^{p_{+}\left(B_{2^{-j}}(x_{0})\right)/p_{-}} \, \dint x \\
&=& \left( \frac{1}{k^{\alpha}} \right)^{p_{+}\left(B_{2^{-k}}(x_{0})\right)/p_{-}} \, \left( \mu\left( B_{2^{-k}}(x_{0}) \right)^{p_{-}(B_{2^{-k}}(x_{0})) - p_{+}(B_{2^{-k}}(x_{0}))}  \right)^{1/p_{-}} \\
& & + \sum_{j=j_{0}}^{k-1} \left( \frac{1}{j^{\alpha}} \right)^{p_{+}\left(B_{2^{-j}}(x_{0})\right)/p_{-}} \, \left( \mu\left( B_{2^{-j}}(x_{0}) \right)^{p_{-}(B_{2^{-j}}(x_{0})) - p_{+}(B_{2^{-j}}(x_{0}))}  \right)^{1/p_{-}}.
\end{eqnarray*}
Using Lemma \ref{lem log-holder}, we have that $\mu\left( B_{2^{-j}}(x_{0}) \right)^{p_{-}(B_{2^{-j}}(x_{0})) - p_{+}(B_{2^{-j}}(x_{0}))} \leq C$. Therefore
\begin{eqnarray*}
\varrho_{p(\cdot)}(s_{k}) \leq C^{1/p_{-}} \, \frac{1}{k^{\alpha}} + C^{1/p_{-}} \, \sum_{j=j_{0}}^{k-1} \frac{1}{j^{\alpha}} \leq C^{1/p_{-}} \sum_{j=1}^{\infty} \frac{1}{j^{\alpha}} < \infty.
\end{eqnarray*}
Let us denote $\beta := C^{1/p_{-}} \sum_{j=1}^{\infty} \frac{1}{j^{\alpha}}$. Naturally, $\beta > 1$, so by \eqref{modular ineq 1},
$$
\varrho_{p(\cdot)}\left( \beta^{-1/p_{-}} \, s_{k} \right) \leq \beta^{-1} \varrho_{p(\cdot)}(s_{k}) \leq 1. 
$$
By the norm-modular unit ball property (see Theorem \ref{norm-modular for Lp}), $\|\beta^{-1/p_{-}} \, s_{k}\|_{p(\cdot)} \leq 1$, that is, $\|s_{k}\|_{p(\cdot)} \leq \beta^{1/p_{-}}$. Moreover, the right-hand side does not depend on $k$, therefore
$$
\sup_{k > j_{0}} \|s_{k}\|_{p(\cdot)} \leq \beta^{1/p_{-}}.
$$

We will show that
$$
\int_{0}^{\mu\left( B_{2^{-j_{0}}}(x_{0}) \right)} \left( \frac{s_{k}^{*}(t)}{t^{-1/p_{-}}} \right)^{v} \, \frac{\dint t}{t} \to \infty \qquad \left(k \to \infty \right)
$$
for all $1 \leq v < p_{-}$. It can be supposed that $j_{0} \geq 2$. Then $b_{j_{0}} < b_{j_{0}+1} < \ldots < b_{k}$. Since $s_{k}$ is a step function, we have that
$$
s_{k}^{*}(t) = b_{k} \, \chi_{\left[0,\mu\left( B_{2^{-k}}(x_{0}) \right)\right)}(t) + \sum_{j=j_{0}}^{k-1} b_{k-1+j_{0}-j} \, \chi_{ \left[ \mu\left( B_{2^{-(k+j_{0}-j)}}(x_{0}) \right),  \mu\left( B_{2^{-(k-1+j_{0}-j)}}(x_{0}) \right) \right) }(t).
$$
Since $v \geq 1$, we obtain that
\begin{eqnarray*}
\lefteqn{ \int_{0}^{\mu\left( B_{2^{-j_{0}}}(x_{0}) \right)} \left( \frac{s_{k}^{*}(t)}{t^{-1/p_{-}(\Omega)}} \right)^{v} \, \frac{\dint t}{t} } \\
&\geq& \int_{0}^{\mu\left( B_{2^{-j_{0}}}(x_{0}) \right)} \left( b_{k}^{v} \chi_{\left[0,\mu\left( B_{2^{-k}}(x_{0}) \right)\right)}(t) \right. \\
& & \left. + \sum_{j=j_{0}}^{k-1} b_{k-1+j_{0}-j}^{v} \, \chi_{\left[\mu\left( B_{2^{-(k+j_{0}-j)}}(x_{0}) \right)  ,\mu\left( B_{2^{-(k-1+j_{0}-j)}}(x_{0}) \right)\right)}(t) \right) \, t^{\frac{v}{p_{-}}-1} \, \dint t\\
&=& \int_{0}^{\mu\left( B_{2^{-k}}(x_{0}) \right)} \left(\frac{\mu\left( B_{2^{-k}}(x_{0}) \right)^{-1}}{k^{\alpha}} \right)^{v/p_{-}} \, t^{\frac{v}{p_{-}}-1} \, \dint t \\
& & + \sum_{j=j_{0}}^{k-1} \int_{\mu\left( B_{2^{-(k+j_{0}-j)}}(x_{0}) \right)}^{\mu\left( B_{2^{-(k-1+j_{0}-j)}}(x_{0}) \right)} \left( \frac{\mu\left( B_{2^{-(k-1+j_{0}-j)}}(x_{0}) \right)^{-1}}{(k-1+j_{0}-j)^{\alpha}} \right)^{v/p_{-}} \, t^{\frac{v}{p_{-}}-1} \, \dint t \\
&=& \left(\frac{\mu\left( B_{2^{-k}}(x_{0}) \right)^{-1}}{k^{\alpha}} \right)^{v/p_{-}} \, \frac{p_{-}}{v} \, \mu\left( B_{2^{-k}}(x_{0}) \right)^{v/p_{-}} \\
& & + (2^{dv/p_{-}}-1) \, \frac{p_{-}}{v} \, \sum_{j=j_{0}}^{k-1} \left( \frac{\mu\left( B_{2^{-(k-1+j_{0}-j)}}(x_{0}) \right)^{-1}}{(k-1+j_{0}-j)^{\alpha}} \right)^{v/p_{-}} \, \mu\left( B_{2^{-(k+j_{0}-j)}}(x_{0}) \right)^{v/p_{-}},
\end{eqnarray*}
where we have used that $\mu\left( B_{2^{-(k+j_{0}-j-1)}}(x_{0}) \right) = 2^d \, \mu\left( B_{2^{-(k+j_{0}-j)}}(x_{0}) \right)$. After simplifying, we have that
\begin{eqnarray*}
\lefteqn{ \int_{0}^{\mu\left( B_{2^{-j_{0}}}(x_{0}) \right)} \left( \frac{s_{k}^{*}(t)}{t^{-1/p_{-}(\Omega)}} \right)^{v} \, \frac{\dint t}{t} } \\
&\geq& \frac{p_{-}}{v} \, \left( \frac{1}{k^{\alpha}} \right)^{v/p_{-}} + \frac{p_{-}}{v} \, \frac{2^{dv/p_{-}}-1}{2^{dv/p_{-}}} \sum_{j=j_{0}}^{k-1} \left( \frac{1}{(k-1+j_{0}-j)^{\alpha}} \right)^{v/p_{-}} \\
&\geq& C_{d,v,p(\cdot)} \, \sum_{l=j_{0}}^{k-1} \left( \frac{1}{l^{\alpha}} \right)^{v/p_{-}} \\
&=& C_{d,v,p(\cdot)} \, \sum_{l=j_{0}}^{k-1} \left( \frac{1}{l} \right)^{\alpha v /p_{-}} \to \infty \qquad \left(k \to \infty \right),
\end{eqnarray*}
because $ \alpha v/p_{-} \leq 1$. This means, that if $k$ is large enough, then
$$
\left( \int_{0}^{\mu\left( B_{2^{-j_{0}}}(x_{0}) \right)} \left( \frac{s_{k}^{*}(t)}{t^{-1/p_{-}(\Omega)}} \right)^{v} \, \frac{\dint t}{t} \right)^{1/v} \not\leq \|s_{k}\|_{p(\cdot)},
$$
which means that $v \geq p_{-}$, and therefore $\uGindexv{L_{p(\cdot)}(\Omega)} \geq p_{-}$.
\end{proof}

We get immediately the following corollary.

\begin{cor}\label{cor ind leb}
Let $\mu(\Omega) < \infty$ and $p(\cdot) \in \cP$ with $1 < p_{-} \leq p_{+} <\infty$. If there exists $x_{0} \in \Omega$, such that $p(x_{0}) = p_{-}$ and $p(\cdot) \in LH_{0}\{ x_{0} \}$, then
$$
\envg(L_{p(\cdot)}(\Omega)) = \left( t^{-1/p_{-}}, p_{-} \right).
$$
\end{cor}

\section{The variable Lorentz space}\label{var-Lor}

Using \eqref{lorentz-equiv}, the space $L_{p(\cdot),q}$ can be defined, where $p(\cdot) \in \cP$ and $0 < q \leq \infty$. The measurable function $f : \Omega \to \R$ belongs to the space $L_{p(\cdot),q}$, if
$$
\|f\|_{L_{p(\cdot),q}} := \begin{cases} \left( \int_{0}^{\infty} u^q \left\| \chi_{\{ |f| > u \}} \right\|_{p(\cdot)}^q \, \frac{\dint u}{u} \right)^{1/q}, & \text{if $0 < q < \infty$}\\
\sup_{u \in (0,\infty)} u \, \| \chi_{\{ |f| > u \}} \|_{p(\cdot)}, & \text{if $q = \infty$}
\end{cases}
$$
is finite. Now, as before, we write only $L_{p(\cdot),q}(\Omega)$, if the domain is important, for example, if the domain is bounded. It follows from the definition, that for all measurable sets $A \subset \Omega$ and $0 < q < \infty$,
\begin{eqnarray}\label{relation}
\|\chi_{A}\|_{L_{p(\cdot),q}} = \left( \int_{0}^{\infty} t^q \| \chi_{\{\chi_{A} > t \}} \|_{p(\cdot)}^q \, \frac{\dint t}{t} \right)^{1/q} = \left( \int_{0}^{1} t^q \| \chi_{A} \|_{p(\cdot)}^q \, \frac{\dint t}{t} \right)^{1/q} = \|\chi_{A}\|_{p(\cdot)} \, q^{-1/q}
\end{eqnarray}
and if $q = \infty$, then $\|\chi_{A}\|_{L_{p(\cdot),\infty}} = \|\chi_{A}\|_{p(\cdot)}$.

\subsection{Lower estimate for $\egv{L_{p(\cdot),q}}$}

\begin{prop}\label{lower lor}
If $p_{+} < \infty$, then for all $0 < q \leq \infty$,
$$
\egv{L_{p(\cdot),q}}(t) \geq \sup\left\{ \|\chi_{A_{t}}\|_{L_{p(\cdot),q}}^{-1} : \mu(A_{t}) = t \right\}, \qquad t>0.
$$
\end{prop}
\begin{proof}
  The proof is similar to the proof of Proposition \ref{thm lower leb}. Let $t > 0$, $s> t$ be fixed, and choose $A_{s} \subset \R^d$ with $\mu(A_{s}) = s$. We consider the functions $\varphi_{s,A_{s}} := \|\chi_{A_{s}}\|_{L_{p(\cdot),q}}^{-1} \chi_{A_{s}}$. Then $\|\varphi_{s,A_{s}}\|_{L_{p(\cdot),q}} = 1$ and $\varphi_{s,A_{s}}^{\ast}(t) = \|\chi_{A_{s}}\|_{L_{p(\cdot),q}}^{-1}$. Using \eqref{relation} and \eqref{sup rec ineq}, we get the following lower estimate for $0 < q < \infty$ (the case $q=\infty$ follows analogously):
\begin{eqnarray*}
\egv{L_{p(\cdot),q}}(t) &=& \sup_{\|f\|_{L_{p(\cdot),q}} \leq 1} f^{\ast}(t) \geq \sup_{s > t, \mu(A_{s})=s} \varphi_{s,A_{s}}^{\ast}(t) = \sup_{s > t} \|\chi_{A_{s}}\|_{L_{p(\cdot),q}}^{-1} \\
&\geq& \sup_{s > t, A_{t} \subset A_{s}} \|\chi_{A_{s}}\|_{L_{p(\cdot),q}}^{-1} = q^{1/q} \sup_{s > t, A_{t} \subset A_{s}} \|\chi_{A_{s}}\|_{p(\cdot)}^{-1} = q^{1/q} \, \| \chi_{A_{t}}\|_{p(\cdot)}^{-1} \\
&=&  \| \chi_{A_{t}}\|_{L_{p(\cdot),q}}^{-1}, \qquad t > 0 
\end{eqnarray*}
holds for all set $A_{t}$ with measure $t$. This implies that
$$
\egv{L_{p(\cdot),q}}(t) \geq \sup\left\{ \|\chi_{A_{t}}\|_{L_{p(\cdot),q}}^{-1} : \mu(A_{t}) = t \right\}, \qquad t>0,
$$
which proves the proposition.
\end{proof}

\subsection{Upper estimate for $\egv{L_{p(\cdot),q}}$}

\begin{thm}\label{thm upper lorentz old}
Let $p(\cdot) \in \cP$ with $p_{+} < \infty$ and $0 < q \leq \infty$. If there exists $t_{0} >0$ and a set $A_{t_{0}}$ with measure $t_{0}$, such that $p(x) = p_{-}$ $(x \in A_{t_{0}})$, then
$$
\egv{L_{p(\cdot),q}}(t) \leq \sup \left\{ \|\chi_{A_{t}}\|_{L_{p(\cdot),q}}^{-1}: \mu(A_{t}) = t \right\}, \quad 0 <t < \min\{ 1, t_{0} \}.
$$
\end{thm}
\begin{proof}
  Let $0 < t < \min\{1,t_{0}\}$ and denote $\alpha := \sup \left\{ \|\chi_{A_{t}}\|_{L_{p(\cdot),q}}^{-1}: \mu(A_{t}) = t \right\}$. For convenience we may assume that $q<\infty$, but the necessary modifications otherwise are obvious. We proceed by contradiction and suppose, that there exists a function $f \in L_{p(\cdot),q}$, with $\|f\|_{L_{p(\cdot),q}} \leq 1$, such that $\mu\left( \left\{ |f| > \alpha \right\} \right) > t$. We can suppose that $f \in L_{p(\cdot),q}$, such that $t < \mu\left( \left\{ |f| > \alpha \right\} \right) < 1$. We have again, that $|f| \geq \alpha \chi_{ \left\{ |f| > \alpha \right\}}$. If $0 < u < \alpha$, then $\{ |f| > u \} \supset \{ |f| > \alpha \}$ and therefore $\chi_{\{ |f| > u \}} \geq \chi_{\{ |f| > \alpha \}}$. Then we obtain, that
$$
\mu\left( \{ |f| > u \} \right) \geq \mu\left( \{ |f| > \alpha \} \right) > t.
$$
Since $\mu(\{|f| > \alpha \}) < 1$, we get from \eqref{char norm 1}, that
$$
\|\chi_{\{ |f| > \alpha \}}\|_{p(\cdot)} \geq \mu\left( \{ |f| > \alpha \} \right)^{1/p_{-}} > t^{1/p_{-}}.
$$
By the monotonicity of the $\|\cdot\|_{p(\cdot)}$-norm (see Lemma \ref{theorems for Lp}), for all $0 < u < \alpha$, $\|\chi_{ \{|f| > u \}}\|_{p(\cdot)} > t^{1/p_{-}}$. Similarly, as before, because of $p(x) = p_{-}$ $(x \in A_{t_{0}})$, we have that $\|\chi_{A_{t_{0}}}\|_{p(\cdot)} \leq t_{0}^{1/p_{-}}$ and by \eqref{relation}, $\|\chi_{A_{t_{0}}}\|_{L_{p(\cdot),q}}^{-1} = q^{1/q} \, \| \chi_{A_{t_{0}}} \|_{p(\cdot)}^{-1} \geq q^{1/q} \, t_{0}^{-1/p_{-}}$. Moreover, for all $t < t_{0}$ and sets $A_{t} \subset A_{t_{0}}$, $\|\chi_{A_{t}}\|_{L_{p(\cdot),q}}^{-1} \geq q^{1/q} \, t^{-1/p_{-}}$, that is,
$$
\alpha = \sup \left\{ \|\chi_{A_{t}}\|_{L_{p(\cdot),q}}^{-1} : \mu(A_{t})=t \right\} \geq q^{1/q} \, t^{-1/p_{-}}.
$$
Hence
\begin{eqnarray*}
1 &\geq& \|f\|_{L_{p(\cdot),q}} = \left( \int_{0}^{\alpha} u^{q} \|\chi_{\{ |f| > u \}}\|_{p(\cdot)}^{q} \, \frac{\dint u}{u} + \int_{\alpha}^{\infty} u^{q} \|\chi_{\{ |f| > u \}}\|_{p(\cdot)}^{q} \, \frac{\dint u}{u} \right)^{1/q}\\
&\geq& \left( \int_{0}^{\alpha} u^{q} \|\chi_{\{ |f| > u \}}\|_{p(\cdot)}^{q} \, \frac{\dint u}{u} \right)^{1/q} > \left( \int_{0}^{\alpha} u^{q} \, t^{q/p_{-}} \, \frac{\dint u}{u} \right)^{1/q} \\
&=& t^{1/p_{-}} \cdot \alpha \cdot q^{-1/q} \geq t^{1/p_{-}} \cdot q^{1/q} \cdot t^{-1/p_{-}} \cdot q^{-1/q} \\
&=& 1,
\end{eqnarray*}
so we have that $1 > 1$, which is a contradiction.
\end{proof}
Using Proposition \ref{lower lor} and Theorem \ref{thm upper lorentz old}, we have

\begin{cor}\label{cor old-2}
Let $p(\cdot) \in \cP$ with $p_{+} < \infty$ and $0 < q \leq \infty$. If there exists $t_{0} >0$ and a set $A_{t_{0}}$ with measure $t_{0}$, such that $p(x) = p_{-}$ $(x \in A_{t_{0}})$, then
$$
\egv{L_{p(\cdot),q}}(t) = \sup \left\{ \|\chi_{A_{t}}\|_{L_{p(\cdot),q}}^{-1}: \mu(A_{t}) = t \right\}, \quad 0 <t < \min\{ 1, t_{0} \}.
$$
\end{cor}

\begin{remark}
If $p(\cdot) = p$, where $0 < p < \infty$ is a constant and $0 < q < \infty$, then by \eqref{equiv norm2},
\begin{eqnarray*}
\|\chi_{A_{t}}\|_{L_{p(\cdot),q}} &=& \left( \int_{0}^{\infty} u^q \|\chi_{\{ \chi_{A_{t}} > u \}}\|_{p(\cdot)}^{q} \, \frac{\dint u}{u} \right)^{1/q} = \left( \int_{0}^{\infty} u^q \|\chi_{\{ \chi_{A_{t}} > u \}}\|_{p}^{q} \, \frac{\dint u}{u} \right)^{1/q} \\
&=& \|\chi_{A_{t}}\|_{\widetilde{L}_{p,q}} = q^{-1/q} \, t^{1/p},
\end{eqnarray*}
we have that
\begin{eqnarray*}
\egv{L_{p(\cdot),q}}(t) = \sup \left\{ \|\chi_{A_{t}}\|_{L_{p(\cdot),q}}^{-1} : \mu(A_{t}) = t \right\} = \sup \left\{ q^{1/q} \, t^{-1/p} : \mu(A_{t}) = t \right\} = q^{1/q} \, t^{-1/p}, 
\end{eqnarray*} 
which is not the classical result. But if $p(\cdot) = p$, then $\|\cdot\|_{L_{p(\cdot),q}} = \|\cdot\|_{\widetilde{L}_{p,q}}$, which was only an equivalent norm with the norm $\|\cdot\|_{L_{p,q}}$ and we have that $\|\cdot\|_{\widetilde{L}_{p,q}} = p^{-1/q} \, \|\cdot\|_{L_{p,q}}$. Hence,
$$
p^{-1/q} \, \egv{L_{p(\cdot),q}}(t) = p^{-1/q} \, q^{1/q} \, t^{-1/p} = \left( \frac{q}{p} \right)^{1/q} \, t^{-1/p} = \egv{L_{p,q}}(t),
$$
so we recover the classical result.
\end{remark}

\begin{remark}
The space $L_{p(\cdot),q}$ is not rearrangement-invariant for arbitrary $p(\cdot) \in \cP$ and $0 < q \leq \infty$ satisfying the assumptions of Corollary \ref{cor old-2}. Similarly, as in the case of variable Lebesgue spaces (see Remark \ref{remark-fund-leb}), this result can be seen as the extension of our result connecting the growth envelope function $\egX$ and fundamental function $\varphi_X$ in rearrangement-invariant spaces, see Remark \ref{fund}, to more general spaces.
\end{remark}

The condition can be weakened if we suppose that $\mu(\Omega) < \infty$. To this end, first notice that if $\mu(\Omega) < \infty$ and $r(\cdot) \leq p(\cdot)$, then for all $0<q\leq \infty$, $L_{p(\cdot),q}(\Omega) \hookrightarrow L_{r(\cdot),q}(\Omega)$. Indeed, since $\mu(\Omega) < \infty$, $\|\cdot\|_{r(\cdot)} \leq c \, \|\cdot\|_{p(\cdot)}$, therefore
\begin{eqnarray*}
\|f\|_{L_{r(\cdot),q}} &=& \left( \int_{0}^{\infty} u^{q} \|\chi_{ \{ |f| > u \} }\|_{r(\cdot)}^{q} \frac{\dint u}{u} \right)^{1/q} \\
&\leq& c \, \left( \int_{0}^{\infty} u^{q} \|\chi_{ \{ |f| > u \} }\|_{p(\cdot)}^{q} \frac{\dint u}{u} \right)^{1/q} = c \, \|f\|_{L_{p(\cdot),q}}.
\end{eqnarray*}

\begin{remark}\label{remark-fb}
If $\mu(\Omega) < \infty$, then $\egv{L_{p(\cdot),q}(\Omega)} \leq c \, t^{-1/p_{-}}$. Indeed, since $p_{-} \leq p(\cdot)$, therefore $L_{p(\cdot),q}(\Omega) \hookrightarrow \widetilde{L}_{p_{-},q}(\Omega) = L_{p_{-},q}(\Omega)$, that is,
\begin{equation*}\label{upper2}
\egv{L_{p(\cdot),q}(\Omega)}(t) \leq c \, \egv{\widetilde{L}_{p_{-},q}(\Omega)}(t) \sim t^{-1/p_{-}}.
\end{equation*}
\end{remark}

Moreover, if we take the condition, that $p(\cdot)$ is locally log-Hölder continuous at a point $x_{0}$, where $p(x_{0}) = p_{-}$, then we have that $\egv{L_{p(\cdot),q}(\Omega)}(t) \sim t^{-1/p_{-}}$.

\begin{thm}\label{thm-le}
Let $p(\cdot) \in \cP$ with $p_{+} < \infty$ and $0 < q \leq \infty$. If there exists $x_{0} \in \Omega$, such that $p(x_{0}) = p_{-}$ and $p(\cdot) \in LH_{0}\{ x_{0} \}$, then there exists $\varepsilon > 0$, such that
$$
\egv{L_{p(\cdot),q}}(t) \geq c \, t^{-1/p_{-}}, \qquad 0 < t < \varepsilon.
$$
\end{thm}
\begin{proof}
For convenience, suppose that $q < \infty$. If $q = \infty$, then the proof is analogous. Since the set $\Omega$ is open, there exists $j_{0} \in \N$, for which $B_{2^{-j_{0}}}(x_{0}) \subset \Omega$. For $j \geq j_0$ consider
$$
a_{j} := q^{1/q} \, \mu\left( B_{2^{-j}}(x_{0}) \right)^{-1/p_{-}}, \quad f_{j}(x) := a_{j} \, \chi_{B_{2^{-j}}(x_{0})}(x), \qquad x \in \Omega. 
$$
For all $u < a_{j}$, $\varrho_{p(\cdot)}\left( \chi_{\{ f_{j} > u\}} \right) = \mu \left( \{ f_{j} > u \} \right) = \mu\left( B_{2^{-j}}(x_{0}) \right)$, so by \eqref{char norm 1} and \eqref{char norm 2},
$$
\| \chi_{ \{f_{j} > u \}} \|_{p(\cdot)} \leq
\begin{cases}
\mu\left( B_{2^{-j}}(x_{0}) \right)^{1/p_{+}(B_{2^{-j}}(x_{0}))}, &\text{if $\mu(B_{2^{-j}}(x_{0})) \leq 1$},\\
\mu\left( B_{2^{-j}}(x_{0}) \right)^{1/p_{-}}, &\text{if $\mu(B_{2^{-j}}(x_{0})) > 1$}.
\end{cases}
$$
If $\mu(B_{2^{-j}}(x_{0})) \leq 1$, then
$$
\|f_{j}\|_{L_{p(\cdot),q}} = \left( \int_{0}^{a_{j}} u^{q} \, \|\chi_{\{f_{j} > u \}}\|_{p(\cdot)}^{q} \, \frac{\dint u}{u} \right)^{1/q} \leq \mu\left( B_{2^{-j}}(x_{0}) \right)^{1/p_{+}(B_{2^{-j}}(x_{0}))} \, q^{-1/q} \, a_{j}.
$$
Since for all $j \geq j_{0}$, $p_{-} = p_{-}(B_{2^{-j}}(x_{0}))$, by Lemma \ref{lem log-holder},
\begin{eqnarray*}
\lefteqn{\mu\left( B_{2^{-j}}(x_{0}) \right)^{1/p_{+}(B_{2^{-j}}(x_{0}))} \, \mu\left( B_{2^{-j}}(x_{0}) \right)^{-1/p_{-}} } \\
&=& \mu\left( B_{2^{-j}}(x_{0}) \right)^{\frac{p_{-}\left(B_{2^{-j}}(x_{0})\right)-p_{+}\left(B_{2^{-j}}(x_{0})\right)}{p_{+}\left(B_{2^{-j}}(x_{0})\right) \, p_{-} }}  \\
&\leq& \left[ \mu\left(B_{2^{-j}}(x_{0})\right)^{p_{-}(B_{2^{-j}}(x_{0})) - p_{+}(B_{2^{-j}}(x_{0}))} \right]^{1/(p_{-})^2} \leq C^{1/(p_{-})^2}.
\end{eqnarray*}
We obtain that if $\mu(B_{2^{-j}}(x_{0})) \leq 1$, then $\|f_{j}\|_{L_{p(\cdot),q}} \leq C^{1/(p_{-}^2)}$. By the definition of $a_{j}$ we get that if $\mu(B_{2^{-j}}(x_{0})) > 1$, then $\|f_{j}\|_{L_{p(\cdot),q}} \leq 1$. This means that the $L_{p(\cdot),q}$-norm of the function
$$
\varphi_{j}(x) :=
\begin{cases}
C^{-1/(p_{-})^2} \, f_{j}(x), &\text{if $\mu(B_{2^{-j}}(x_{0})) \leq 1$},\\
f_{j}(x), &\text{if $\mu(B_{2^{-j}}(x_{0})) > 1$},
\end{cases} \qquad x \in \Omega,
$$
is at most $1$. By \eqref{char-csillag}, for all $0 < t < \mu\left( B_{2^{-j}}(x_{0}) \right)$,
$$
\varphi_{j}^{*}(t) =
\begin{cases}
C^{-1/(p_{-})^2} \, q^{1/q} \, \mu\left( B_{2^{-j}}(x_{0}) \right)^{-1/p_{-}}, &\text{if $\mu(B_{2^{-j}}(x_{0})) \leq 1$},\\
q^{1/q} \, \mu\left( B_{2^{-j}}(x_{0}) \right)^{-1/p_{-}}, &\text{if $\mu(B_{2^{-j}}(x_{0})) > 1$}.
\end{cases}
$$
We have for any $0 < h < 1$,
\begin{eqnarray*}
\egv{L_{p(\cdot),q}}\left( h \, \mu\left( B_{2^{-j}}(x_{0}) \right) \right) &\geq& \varphi_{j}^{*}\left( h \, \mu\left( B_{2^{-j}}(x_{0}) \right) \right) \\
&=& C_{p_{-},p_{+},q} \, \mu\left( B_{2^{-j}}(x_{0}) \right)^{-1/p_{-}}.
\end{eqnarray*}
Since the function $\egv{L_{p(\cdot),q}}$ is decreasing (see Proposition~\ref{properties:EG}), we have that
\begin{eqnarray*}
\egv{L_{p(\cdot),q}}\left( \mu\left(B_{2^{-j}}(x_{0})\right) \right) &=& \inf \left\{ \egv{L_{p(\cdot),q}}(h \, \mu\left( B_{2^{-j}}(x_{0}) \right) ) : 0 < h < 1 \right\} \\
&\geq& C_{p_{-},p_{+},q} \, \mu\left( B_{2^{-j}}(x_{0}) \right)^{-1/p_{-}},
\end{eqnarray*}
for all $j \geq j_{0}$. This yields
$$
\egv{L_{p(\cdot),q}}(t) \geq  C_{p_{-},p_{+},q} \, t^{-1/p_{-}}, \qquad 0  < t < \varepsilon,
$$
where $\varepsilon := \mu\left( B_{2^{-j_{0}}}(x_{0})\right)$.
\end{proof}

As a consequence of Remark \ref{remark-fb} and Theorem \ref{thm-le} we obtain

\begin{cor}\label{cor equiv lorentz}
Let $\mu(\Omega) < \infty$, $p(\cdot) \in \cP$ with $p_{+} < \infty$ and $0 < q \leq \infty$. Suppose that there exists $x_{0} \in \Omega$, such that $p(x_{0}) = p_{-}$ and $p(\cdot) \in LH_{0}\{ x_{0} \}$. Then there exists $\varepsilon > 0$, such that
$$
\egv{L_{p(\cdot),q}(\Omega)}(t) \sim t^{-1/p_{-}}, \qquad 0 < t < \varepsilon.
$$
\end{cor}

Similarly, as in the case of the variable Lebesgue spaces, the condition $\mu(\Omega) < \infty$ can be omitted, if we assume, that $p(\cdot)$ is locally log-Hölder continuous.

\begin{thm}
Let $p(\cdot) \in \cP$ with $p_{+} < \infty$, $p(\cdot) \in LH_{0}$ and $0 < q \leq \infty$. If there exists $x_{0} \in \Omega$, such that $p(x_{0}) = p_{-}$, then
$$
\egv{L_{p(\cdot),q}}(t) \leq C \, t^{-1/p_{-}}, \qquad 0 < t \leq 1.
$$
\end{thm}
\begin{proof}
The proof is similar as the proof of Theorem \ref{thm upper better}. Let $t \in \left(0, 1\right]$ be fixed and let us denote
$$
\alpha := c_{2} \, \sup \left\{ \|\chi_{A_{t}}\|_{L_{p(\cdot),q}}^{-1} : A_{t} \subset \R^d \, \mbox{cube (or ball)}, \, \mu(A_{t}) = t  \right\},
$$
where the constant $c_{2}$ is equal with the constant in \eqref{constan c2}. As in the proof of the Theorem \ref{thm upper lorentz old}, we have that for all $0 < u < \alpha$, $\|\chi_{ \{|f| > u \}}\|_{p(\cdot)} > t^{1/p_{-}}$. By \eqref{relation} and \eqref{constan c2},
\begin{eqnarray}
\alpha &=& c_{2} \, q^{1/q} \, \sup \left\{ \|\chi_{A_{t}}\|_{p(\cdot)}^{-1} : A_{t} \subset \R^d \, \mbox{cube (or ball)}, \, \mu(A_{t}) = t  \right\} \n\\
&\geq& c_{2} \, q^{1/q} \, \sup \left\{ \|\chi_{A_{t}}\|_{p(\cdot)}^{-1} : A_{t} \subset \R^d \, \mbox{cube (or ball)}, \, \mu(A_{t}) = t, \, p_{-}(A_{t}) = p_{-}  \right\} \n\\
&\geq& c_{2} \, q^{1/q} \, c_{2}^{-1} \, t^{-1/p_{-}}. \label{sup-2}
\end{eqnarray}
Thus, by \eqref{sup-2},
\begin{eqnarray*}
1 &\geq& \|f\|_{L_{p(\cdot),q}} \geq \left( \int_{0}^{\alpha} u^{q} \|\chi_{\{ |f| > u \}}\|_{p(\cdot)}^{q} \, \frac{\dint u}{u} \right)^{1/q} \\
&>& \left( \int_{0}^{\alpha} u^{q} \, t^{q/p_{-}} \, \frac{\dint u}{u} \right)^{1/q} =  t^{1/p_{-}} \, \alpha \, q^{-1/q} \geq 1,
\end{eqnarray*}
which is a contradiction. Using \eqref{sup upper}, we have that
\begin{eqnarray*}
\egv{L_{p(\cdot),q}}(t) &\leq& c_{2} \, \sup \left\{ \|\chi_{A_{t}}\|_{L_{p(\cdot),q}}^{-1} : A_{t} \subset \R^d \, \mbox{cube (or ball)}, \, \mu(A_{t}) = t  \right\} \\
&=& c_{2} \, q^{1/q} \, \sup \left\{ \|\chi_{A_{t}}\|_{p(\cdot)}^{-1} : A_{t} \subset \R^d \, \mbox{cube (or ball)}, \, \mu(A_{t}) = t \right\}\\
&\leq& q^{1/q} \, \frac{c_{2}}{c_{1}} \,  t^{-1/p_{-}},
\end{eqnarray*}
which finishes the proof.
\end{proof}

\begin{cor}\label{cor equiv lor}
Let $p(\cdot) \in \cP$ with $p_{+} < \infty$, $p(\cdot) \in LH_{0}$ and $0 < q \leq \infty$. If there exists $x_{0} \in \Omega$, such that $p(x_{0}) = p_{-}$, then there exists $\varepsilon > 0$, such that
$$
\egv{L_{p(\cdot),q}}(t) \sim t^{-1/p_{-}}, \qquad 0 < t < \varepsilon.
$$
\end{cor}

\subsection{Additional index}

Now let us consider the additional index.

\begin{thm}
Let $\mu(\Omega) < \infty$, $p(\cdot) \in \cP$ with $p_{+} < \infty$ and $1 < q \leq \infty$. If there exists $x_{0} \in \Omega$, such that $p(x_{0}) = p_{-}$ and $p(\cdot) \in LH_{0}\{ x_{0} \}$, then $\uGindexv{L_{p(\cdot),q}(\Omega)}=q$.
\end{thm}
\begin{proof}
The upper estimate can be reached again by the help of Corollary \ref{cor equiv lorentz}, the embedding $L_{p(\cdot),q}(\Omega) \hookrightarrow L_{p_{-},q}(\Omega)$ and Proposition \ref{index est}.

For the lower estimate, suppose that $q < \infty$ (if $q = \infty$, then the proof is similar). We will show, that for all $1 < v < q$, the inequality
$$
\left(\int_{0}^{\varepsilon}  \left[t^{1/p_{-}} \, f^{*}(t) \right]^{v} \, \frac{\dint t}{t}  \right)^{1/v} \leq c \, \|f\|_{L_{p(\cdot),q}}
$$
does not hold for all functions $f$ from $L_{p(\cdot),q}(\Omega)$. Since $1/v > 1/q$, there exists $\alpha > 0$, such that $p_{-}/q < \alpha < p_{-}/v$. Let us consider the same sequence of functions $s_{k}$ as in the proof of Theorem \ref{index var leb}. First, we will show that $\sup_{k > j_{0}} \|s_{k}\|_{L_{p(\cdot),q}} < \infty$, where $j_{0}$ satisfies that $B_{2^{-j_{0}}}(x_{0}) \subset \Omega$. We can suppose that $\mu\left( B_{2^{-j_{0}}}(x_{0}) \right) < 1$. For all $b_j < u < b_{j+1}$ $(j=j_{0},\ldots,k-1)$, by \eqref{char norm 1},
$$
\|\chi_{\{s_k > u \}}\|_{p(\cdot)} \leq \mu\left( B_{2^{-j}}(x_{0}) \right)^{1/p_{+}(B_{2^{-j}}(x_{0}))}
$$
and therefore
\begin{eqnarray*}
\|s_{k}\|_{L_{p(\cdot),q}} &=& \left( \int_{0}^{\infty} u^q \, \|\chi_{\{ s_{k} > u \} }\|_{p(\cdot)}^{q} \, \frac{\dint u}{u} \right)^{1/q} = \left[ \sum_{j=j_{0}}^{k-1} \int_{b_{j}}^{b_{j+1}} u^{q} \, \|\chi_{\{ s_{k} > u \}}\|_{p(\cdot)}^{q} \, \frac{\dint u}{u}  \right]^{1/q} \\
&\leq&  \left[ \sum_{j=j_{0}}^{k-1} \mu\left( B_{2^{-j}}(x_{0}) \right)^{\frac{q}{p_{+}(B_{2^{-j}}(x_{0}))}} \, q^{-1} \, b_{j+1}^{q}  \right]^{1/q}\\
&=& q^{-1/q} \, 2^{-d/p_{-}} \, \left[ \sum_{j=j_{0}}^{k-1} \mu\left( B_{2^{-j}}(x_{0}) \right)^{\frac{q}{p_{+}(B_{2^{-j}}(x_{0}))} - \frac{q}{p_{-}}} \, \frac{1}{(j+1)^{q \alpha / p_{-}}}  \right]^{1/q},
\end{eqnarray*}
where we have used, that $\mu(B_{2^{-(j+1)}}(x_{0})) = 2^{-d} \, \mu(B_{2^{-j}}(x_{0}))$. Here by Lemma \ref{lem log-holder},
\begin{eqnarray*}
\lefteqn{\mu\left( B_{2^{-j}}(x_{0}) \right)^{\frac{q \, p_{-} - q \, p_{+}(B_{2^{-j}}(x_{0}))}{p_{+}(B_{2^{-j}}(x_{0})) \, p_{-}}} }\\
&=& \left[ \mu\left( B_{2^{-j}}(x_{0})\right)^{p_{-}(B_{2^{-j}}(x_{0})) - p_{+}(B_{2^{-j}}(x_{0}))} \right]^{\frac{q}{p_{+}(B_{2^{-j}}(x_{0})) \, p_{-}}} \\
&\leq& \left[ \mu\left( B_{2^{-j}}(x_{0})\right)^{p_{-}(B_{2^{-j}}(x_{0})) - p_{+}(B_{2^{-j}}(x_{0}))} \right]^{\frac{q}{(p_{-})^2}}  \\
&\leq& C^{q/(p_{-})^2}.
\end{eqnarray*}
Hence
$$
\|s_{k}\|_{L_{p(\cdot),q}} \leq q^{-1/q} \, 2^{-d/p_{-}} C^{1/(p_{-})^2} \, \left[ \sum_{j=0}^{\infty} \left( \frac{1}{j+1} \right)^{\frac{q \alpha}{p_{-}}}  \right]^{1/q} < \infty,
$$
because of $q \, \alpha/p_- > 1$.

By the same way, as in the proof of Theorem \ref{index var leb}, we have that
$$
\int_{0}^{\mu\left( B_{2^{-j_{0}}}(x_{0}) \right)} \left( \frac{s_{k}^{*}(t)}{t^{-1/p_{-}}} \right)^{v} \, \frac{\dint t}{t} \geq \sum_{l=j_{0}}^{k-1} \left( \frac{1}{l} \right)^{\alpha v /p_{-}} \to \infty \qquad \left(k \to \infty \right),
$$
since $\alpha \, v / p_{-} < 1$. That is, $v \geq q$, and therefore $\uGindexv{L_{p(\cdot),q}(\Omega)} = q$, which proves the theorem.
\end{proof}

We obtain the following corollary.

\begin{cor}\label{cor ind lor}
Let $\mu(\Omega) < \infty$, $p(\cdot) \in \cP$ with $p_{+} < \infty$ and $1 < q \leq \infty$. If there exists $x_{0} \in \Omega$, such that $p(x_{0}) = p_{-}$ and $p(\cdot) \in LH_{0}\{ x_{0} \}$, then
$$
\envg(L_{p(\cdot),q}(\Omega)) = \left( t^{-1/p_{-}}, q \right).
$$
\end{cor}

\section{Applications}\label{appli}

Using the growth envelopes of the function spaces $L_{\pv}$, $L_{\pv,q}$, $L_{p(\cdot)}$ and $L_{p(\cdot),q}$, some \emph{Hardy-type inequalities} and \emph{limiting embeddings} can be obtained.

\subsection{Hardy-type inequalities}

Using the results from \cite[Corollary~11.1, Remark~3.9]{Haroske2007} together with the fact that if 
$\mu(\Omega) < \infty$ we computed for the growth envelopes $\envg(L_{\pv}) = (t^{-1/\min\{p_{1}, \ldots, p_{d} \}}, \min\{p_{1}, \ldots, p_{d} \})$ and $\envg(L_{\pv,q}) = (t^{-1/\min\{p_{1}, \ldots, p_{d} \}}, q)$, we obtain the following Hardy-type inequalities.

\begin{cor}
Let $\varepsilon > 0$ be small, $\varkappa$ be a positive monotonically decreasing function on $(0,\varepsilon]$ and $0<v\leq \infty$. Let $\mu(\Omega) < \infty$, $0 < \pv \leq \infty$ with $\min\{p_{1}, \ldots, p_{d} \} < \infty$. Then
\begin{equation*}
\left(\int_{0}^{\varepsilon}  \left[\varkappa(t) \, t^{1/\min\{p_{1}, \ldots, p_{d} \}} \, f^{*}(t) \right]^{v} \, \frac{\dint t}{t}  \right)^{1/v} \leq c \, \|f\|_{\pv}
\end{equation*}
for some $c > 0$ and all $f \in L_{\pv}(\Omega)$ if, and only if, $\varkappa$ is bounded on $(0,\varepsilon]$ and $v$ satisfies $\min\{p_{1},\ldots, p_{d} \} \leq v \leq \infty$, with the modification
\begin{equation}\label{spec eq2}
\sup_{0 < t < \varepsilon} \varkappa(t) \, t^{1/\min\{p_{1}, \ldots, p_{d} \}} \, f^{*}(t) \leq c \, \|f\|_{\pv}
\end{equation}
if $v = \infty$. In particular, if $\varkappa$ is an arbitrary non-negative function on $(0,\varepsilon]$, then \eqref{spec eq2} holds if, and only if, $\varkappa$ is bounded.
\end{cor}
Concerning the mixed Lorentz spaces we get similar results.

\begin{cor}
Let $\varepsilon > 0$ be small, $\varkappa$ be a positive monotonically decreasing function on $(0,\varepsilon]$ and $0<v\leq \infty$. Let $0 < \pv \leq \infty$ with $\min\{p_{1}, \ldots, p_{d} \} < \infty$ and $0 < q \leq \infty$. Then
\begin{equation*}
\left(\int_{0}^{\varepsilon}  \left[\varkappa(t) \, t^{1/\min\{p_{1}, \ldots, p_{d} \}} \, f^{*}(t) \right]^{v} \, \frac{\dint t}{t}  \right)^{1/v} \leq c \, \|f\|_{L_{\pv,q}}
\end{equation*}
for some $c > 0$ and for all $f \in L_{\pv,q}(\Omega)$ if, and only if, $\varkappa$ is bounded on $(0,\varepsilon]$ and $q \leq v \leq \infty$, with the modification
\begin{equation}\label{spec eq3}
\sup_{0 < t < \varepsilon} \varkappa(t) \, t^{1/\min\{p_{1}, \ldots, p_{d} \}} \, f^{*}(t) \leq c \, \|f\|_{L_{\pv,q}}
\end{equation}
if $v = \infty$. In particular, if $\varkappa$ is an arbitrary non-negative function on $(0,\varepsilon]$, then \eqref{spec eq3} holds if, and only if, $\varkappa$ is bounded.
\end{cor}

For the variable Lebesgue and Lorentz spaces we need to assume further conditions on the exponent function.

\begin{cor}
Let $\varepsilon > 0$ be small, $\varkappa$ be a positive monotonically decreasing function on $(0,\varepsilon]$ and $0<v \leq \infty$. Let $\mu(\Omega) < \infty$, $p(\cdot) \in \cP$ with $1 < p_{-} \leq p_{+} <\infty$ and suppose that there exists $x_{0} \in \Omega$, such that $p(x_{0}) = p_{-}$ and $p(\cdot) \in LH_{0}\{ x_{0} \}$. Then
\begin{equation*}
\left(\int_{0}^{\varepsilon}  \left[\varkappa(t) \, t^{1/p_{-}} \, f^{*}(t) \right]^{v} \, \frac{\dint t}{t}  \right)^{1/v} \leq c \, \|f\|_{p(\cdot)}
\end{equation*}
for some $c > 0$ and all $f \in L_{p(\cdot)}(\Omega)$ if, and only if, $\varkappa$ is bounded on the interval $(0,\varepsilon]$ and $p_{-} \leq v \leq \infty$, with the modification
\begin{equation}\label{spec eq4}
\sup_{0 < t < \varepsilon} \varkappa(t) \, t^{1/p_{-}} \, f^{*}(t) \leq c \, \|f\|_{p(\cdot)}
\end{equation}
if $v = \infty$. In particular, if $\varkappa$ is an arbitrary non-negative function on the $(0,\varepsilon]$, then \eqref{spec eq4} holds if, and only if, $\varkappa$ is bounded.
\end{cor}

\begin{cor}
Let $\varepsilon > 0$ be small, $\varkappa$ be a positive monotonically decreasing function on $(0,\varepsilon]$ and $0<v \leq \infty$. Let $\mu(\Omega) < \infty$, $1 < q \leq \infty$, $p(\cdot) \in \cP$ with $p_{+} < \infty$ and suppose that there exists $x_{0} \in \Omega$, such that $p(x_{0}) = p_{-}$ with $p(\cdot) \in LH_{0}\{ x_{0} \}$. Then
\begin{equation*}
\left(\int_{0}^{\varepsilon}  \left[\varkappa(t) \, t^{1/p_{-}} \, f^{*}(t) \right]^{v} \, \frac{\dint t}{t}  \right)^{1/v} \leq c \, \|f\|_{L_{p(\cdot),q}}
\end{equation*}
for some $c > 0$ and all $f \in L_{p(\cdot),q}(\Omega)$ if, and only if, $\varkappa$ is bounded on $(0,\varepsilon]$ and $q \leq v \leq \infty$, with the modification
\begin{equation}\label{spec eq5}
\sup_{0 < t < \varepsilon} \varkappa(t) \, t^{1/p_{-}} \, f^{*}(t) \leq c \, \|f\|_{L_{p(\cdot),q}}
\end{equation}
if $v = \infty$. In particular, if $\varkappa$ is an arbitrary non-negative function on $(0,\varepsilon]$, then \eqref{spec eq5} holds if, and only if, $\varkappa$ is bounded.
\end{cor}

\subsection{Limiting embeddings}

Using Proposition~\ref{properties:EG}, if there is no constant $c >0$, for which
$$
\sup_{0 < t < \varepsilon} \frac{\egv{X_{1}}(t)}{\egv{X_{2}}(t)} \leq c < \infty,
$$
then $X_{1} \not\hookrightarrow X_{2}$. Using this, Theorem \ref{narrowest embedding} can be proved much easier.

\begin{cor}\label{appl}
Let $\mu(\Omega) < \infty$, and $0 < \pv \leq \infty$ with $0 < \min\{p_{1},\ldots, p_{d}\} < \infty$. If $r > \min\{p_{1}, \ldots, p_{d} \}$, then we have
\begin{enumerate}
\item $L_{\pv}(\Omega) \not\hookrightarrow L_{r}(\Omega)$ and
\item $L_{\pv,q}(\Omega) \not\hookrightarrow L_{r,s}(\Omega)$ for all $0 < q,s \leq \infty$.
\end{enumerate}
\end{cor}
\begin{proof}
Indeed, by Theorem \ref{cor mixed leb} and \eqref{Lp-global},
$$
\frac{\egv{L_{\pv}(\Omega)}(t)}{\egv{L_{r}(\Omega)}(t)} = \frac{t^{1/r}}{t^{1/\min\{p_{1},\ldots,p_{d} \}}} \to \infty, \qquad t \downarrow 0,
$$
that is, $L_{\pv}(\Omega) \not\hookrightarrow L_{r}(\Omega)$. The case of $L_{\pv,q}(\Omega)$ can be proved similarly.
\end{proof}

What is the "smallest" classical Lorentz space $L_{r,s}(\Omega)$, which contains the space $L_{\pv,q}(\Omega)$ ? From the previous outcome and the embedding $L_{\pv,q}(\Omega) \hookrightarrow L_{\min\{p_{1}, \ldots, p_{d} \},q}(\Omega)$, it follows, that $r\leq \min\{p_{1}, \ldots, p_{d} \}$. But what about $s$? By Lemma \ref{lemma emb lorentz}, we have that if $s \geq q$, then
$$
L_{\pv,q}(\Omega) \hookrightarrow L_{\min\{p_{1}, \ldots, p_{d} \},q}(\Omega) \hookrightarrow L_{\min\{p_1, \ldots, p_{d}\},s}(\Omega).
$$
Moreover, we have
\begin{cor}\label{cor emb ind}
Let $0 < \pv \leq \infty$ with $\min\{p_{1}, \ldots, p_{d} \} < \infty$ and $0 < q \leq \infty$. If $\mu(\Omega) < \infty$, then
\begin{enumerate}
\item $L_{\pv}(\Omega) \not\hookrightarrow L_{\min\{p_{1}, \ldots, p_{d} \},s}(\Omega)$ for all $s < \min\{p_{1}, \ldots, p_{d} \}$;
\item $L_{\pv,q}(\Omega) \not\hookrightarrow L_{\min\{p_{1}, \ldots, p_{d}\},u}(\Omega)$ for all $0 < u < q$.
\end{enumerate}
\end{cor} 
\begin{proof}
We will show 1. The case of 2. is analogous. We proceed 1. by contradiction. Suppose, that $s < \min\{p_{1}, \ldots, p_{d} \}$ and $L_{\pv}(\Omega) \hookrightarrow L_{\min\{p_1,\ldots, p_{d} \},s}(\Omega)$. Since
$$
\egv{L_{\pv}(\Omega)}(t) \sim t^{-1/\min\{p_{1}, \ldots, p_{d} \}} \sim \egv{L_{\min\{p_1, \ldots, p_{d} \},s}(\Omega)}(t),
$$
and $L_{\pv}(\Omega) \hookrightarrow L_{\min\{p_1,\ldots, p_{d} \},s}(\Omega)$, by Proposition~\ref{index est}, we obtain that
$$
\uGindexv{L_{\pv}(\Omega)} = \min\{p_{1}, \ldots, p_{d} \} \leq s = \uGindexv{L_{\min\{p_{1}, \ldots, p_{d} \},s}(\Omega)},
$$
which is a contradiction.
\end{proof}
So, we have that the smallest classical Lorentz space containing the space $L_{\pv}(\Omega)$, is $L_{\min\{p_{1}, \ldots, p_{d} \},\min\{p_{1}, \ldots, p_{d} \}}(\Omega)$ and the smallest classical Lorentz space, which contains the mixed Lorentz space $L_{\pv,q}(\Omega)$, is $L_{\min\{p_1, \ldots, p_d\},q}(\Omega)$.

For the variable Lebesgue spaces we have the following corollaries.

\begin{cor}\label{var emb}
Let $p(\cdot) \in \cP$ with $p_{+} <\infty$, $p(\cdot) \in LH_{0}$ and $0 < q \leq \infty$. If there exists $x_{0} \in \Omega$, such that $p(x_{0}) = p_{-}$, then for all $r > p_{-}$ we have that
\begin{enumerate}
\item $L_{p(\cdot)} \not\hookrightarrow L_{r}$ and
\item $L_{p(\cdot),q} \not\hookrightarrow L_{r,s}$ for all $0 < s \leq \infty$.
\end{enumerate}
\end{cor}
\begin{proof}
To verify the claims, let us use Corollary \ref{cor equiv leb} and Corollary \ref{cor equiv lor}. The proof is analogous to the proof of Corollary~\ref{appl}.
\end{proof}
For the embedding $L_{p(\cdot)} \hookrightarrow L_{r}$ (or $L_{p(\cdot),q} \hookrightarrow L_{r,s}$), it is necessary, that $r \leq p_{-}$. If $\mu(\Omega) < \infty$ and $r \leq p_{-}$, then then the smallest space from the spaces $\{ L_{r}(\Omega) : r \leq p_{-} \}$ is $L_{p_{-}}(\Omega) = L_{p_{-},p_{-}}(\Omega)$. Similarly, for fixed $0 < s \leq \infty$, the smallest space from $\{L_{r,s}(\Omega) : r \leq p_{-}\}$ is $L_{p_{-},s}(\Omega)$. Therefore, if $\mu(\Omega) < \infty$ and we look for the smallest classical Lebesgue or classical Lorentz space, which contains $L_{p(\cdot)}(\Omega)$ (or $L_{p(\cdot),q}(\Omega)$, respectively), then the first index must be $p_{-}$. What about the second index? In this case we make use of the additional index which yields the following.

\begin{cor}\label{var emb ind}
Let $\mu(\Omega) < \infty$, $p(\cdot) \in \cP$ with $p_{+} <\infty$ and suppose that there exists $x_{0} \in \Omega$, such that $p(x_{0}) = p_{-}$ and $p(\cdot) \in LH_{0}\{ x_{0} \}$.
\begin{enumerate}
\item If $p_{-} > 1$, then $L_{p(\cdot)}(\Omega) \not\hookrightarrow L_{p_{-},s}(\Omega)$ for all $s < p_{-}$.
\item If $1 < q \leq \infty$, then $L_{p(\cdot),q}(\Omega) \not\hookrightarrow L_{p_{-},u}(\Omega)$ for all $u < q$.
\end{enumerate}
\end{cor}
\begin{proof}
The proof is analogous to the proof of Corollary~\ref{cor emb ind}. For 1., Corollary \ref{cor ind leb} is used and Corollary \ref{cor ind lor} is applied for 2.
\end{proof}
In conclusion, the smallest classical Lorentz space, containing the variable Lebesgue space $L_{p(\cdot)}(\Omega)$ is the space $L_{p_{-},p_{-}}(\Omega)$ and the smallest classical Lorentz space, which contains the variable Lorentz space $L_{p(\cdot),q}(\Omega)$, is $L_{p_{-},q}(\Omega)$.

By Corollaries \ref{appl}, \ref{cor emb ind}, \ref{var emb} and \ref{var emb ind} we have the following necessary conditions for the following embeddings. 

\begin{remark}
The following necessary conditions are obtained.
\begin{enumerate}
\item If $\mu(\Omega) < \infty$ and $0 < \pv \leq \infty$, then
  \begin{enumerate}
  \item from $L_{\pv}(\Omega) \hookrightarrow L_{r}(\Omega)$, we have that $\min\{p_{1}, \ldots, p_{d} \} \geq r$;
  \item for all $0 < s \leq \infty$, from the embedding $L_{\pv,q}(\Omega) \hookrightarrow L_{r,s}(\Omega)$, it follows that $\min\{p_{1}, \ldots, p_{d} \} \geq r$.
  \end{enumerate}
\item If $\mu(\Omega) < \infty$, $0 < \pv \leq \infty$ and $0 < q \leq \infty$, then
  \begin{enumerate}
  \item from $L_{\pv}(\Omega) \hookrightarrow L_{\min\{p_{1}, \ldots, p_{d} \},s}(\Omega)$, we have that $\min\{p_{1}, \ldots, p_{d} \} \leq s$;
  \item for the embedding $L_{\pv,q}(\Omega) \hookrightarrow L_{\min\{p_{1},\ldots,p_{d} \},u}(\Omega)$, it is necessary that $q \leq u$. 
  \end{enumerate}
\item Let $p(\cdot) \in \cP$ with $p_{+} < \infty$, $p(\cdot) \in LH_{0}$ and $0 < q \leq \infty$. If there exists $x_{0}$, such that $p(x_{0}) = p_{-}$, then
  \begin{enumerate}
  \item from $L_{p(\cdot)} \hookrightarrow L_{r}$, we have that $p_{-} \geq r$;
  \item for all $0 < s \leq \infty$, if $L_{p(\cdot),q} \hookrightarrow L_{r,s}$, then $p_{-} \geq r$.
  \end{enumerate}
\item Let $\mu(\Omega) < \infty$, $p(\cdot) \in \cP$ with $p_{+} < \infty$ and suppose that there exists $x_{0} \in \Omega$, such that $p(x_{0}) = p_{-}$ and $p(\cdot) \in LH_{0}\{ x_{0} \}$.
  \begin{enumerate}
  \item If $p_{-} > 1$ and $L_{p(\cdot)}(\Omega) \hookrightarrow L_{p_{-},s}(\Omega)$, then $p_{-} \leq s$.
  \item If $1 < q  \leq \infty$ and $L_{p(\cdot),q}(\Omega) \hookrightarrow L_{p_{-},u}(\Omega)$, then $q \leq u$.
  \end{enumerate}

\end{enumerate}
\end{remark}




\end{document}